\documentclass[12pt,twoside]{amsart}
\usepackage{amssymb,amsmath,amsfonts,amsthm}
\usepackage{mathrsfs}
\usepackage{tgtermes}
\usepackage{mathtools}
\usepackage[X2,T1]{fontenc}
\usepackage[applemac]{inputenc}
\usepackage{cite}
\usepackage{latexsym}
\usepackage[dvips]{graphicx}
\usepackage{comment}
\usepackage{ulem}
\usepackage{enumitem}
\def\R{\mathbb R}

\def\N{\mathbb N}

\def\eps{\varepsilon}
\usepackage{colortbl}

\newcommand{\beqsn}{\arraycolsep1.5pt\begin{eqnarray*}}
	\newcommand{\eeqsn}{\end{eqnarray*}\arraycolsep5pt}
\newcommand{\beqs}{\arraycolsep1.5pt\begin{eqnarray}}
	\newcommand{\eeqs}{\end{eqnarray}\arraycolsep5pt}

\newcommand{\abs}[1]{\lvert#1\rvert}

\newtheorem{Th}{Theorem}[section]
\newtheorem{Rem}[Th]{Remark}
\newtheorem{Ex}[Th]{Example}
\newtheorem{Lemma}[Th]{Lemma}
\newtheorem{Def}[Th]{Definition}
\newtheorem{Prop}[Th]{Proposition}

\newtheorem{hp}{}

\catcode`\@=11

\renewcommand{\section}%
{\setcounter{equation}{0}\@startsection {section}{1}{\z@}{-3.5ex plus -1ex
		minus -.2ex}{2.3ex plus .2ex}{\Large\bf}}

\title{Schr\"odinger ultrahyperbolic equations with singular coefficients}
\author[Garetto]{Claudia Garetto}
\address{Claudia Garetto\\
	School of Mathematical Sciences\\Queen Mary University of London\\
Mile End Road, London E1 4NS\\ UK}
\email{c.garetto@qmul.ac.uk}

\author[Tramontana]{Davide Tramontana}
\address{Davide Tramontana\\
	Dipartimento di Matematica 
    Alma Mater Studiorum\\Universit\`a di Bologna\\
	 Piazza di Porta S. Donato 5\\
    40126 Bologna\\
	Italy}
\email{davide.tramontana4@unibo.it}

\date{}

\thanks{The authors were supported by the EPSRC grant EP/V005529/2. The second author is member of the Research Group GNAMPA of INdAM}
\thanks{On behalf of all authors, the corresponding author states that there is no conflict of interest. This manuscript has no associated data.}
\begin{document}

\begin{abstract}
In this paper we investigate the Cauchy problem for Schr\"odinger ultrahyperbolic equations with singular (less than continuous) coefficients. We prove $H^\infty$ well-posedness in the very weak sense under suitable assumptions of the distributional structure of the coefficients and decay on the lower order terms. Consistency is proven with the classical $H^\infty$-results when the equation coefficients are smooth.
\end{abstract}
\maketitle

\section{Introduction}
Schr\"odinger type equations play a central role in modelling wave phenomena across a wide range of physical contexts, including optics, quantum mechanics, and complex fluid dynamics. 
This paper is devoted to a specific class of these equations, known as \textit{ultrahyperbolic Schr\"odinger} equations, which are characterised by direction-dependent dispersion. This anisotropic dispersive structure makes their mathematical analysis especially subtle and demanding. Their analysis provides insight into how waves and signals propagate, interact, and evolve over time, particularly in media where dispersive effects and material properties exhibit intricate behaviour. Such equations naturally emerge in the study of nonlinear wave phenomena, including multidimensional water-wave models (see \cite{DS,SZ}) and certain completely integrable systems (see \cite{AH}). In detail, we will study the Cauchy problem  
\begin{equation}
\begin{cases}\label{mainprobIntro}
    Pu=f \quad \text{in} \ (0,T] \times \R^n, \\
    u(0,\cdot)=u_0 \quad \text{in} \ \R^n,
\end{cases}
\end{equation}
where $u_0 \in \mathscr{D}'(\mathbb{R}^n)$, $f \in C([0,T];\mathscr{D}'(\mathbb{R}^n))$ and 
\begin{equation}\label{eq.P}
    P=D_t-\sum_{i,j=1}^nD_{x_i}(a_{ij}(x)D_{x_j})- \sum_{k=1}^nb_k(x)D_{x_k}-V(x),
\end{equation}
with $a_{ij},b_k,V \in \mathscr{D}'(\mathbb{R}^n)$, for all $i,j,k=1,\dots,n$, non-smooth coefficients that satisfy suitable assumptions. Throughout the paper we will use the notation $D_t=-i\partial_t$ and $D_{x_j}=-i\partial_{x_j}$. We will assume that the matrix 
$\mathcal{A}=(a_{ij}(x))_{i,j=1}^n$ is real and symmetric and satisfies a non-degeneracy condition. The term ultrahyperbolic in this context refers to the fact that the operator $A$ involves different directions in space that have opposing effects on wave propagation. A typical example in $\mathbb{R}^2$ is given by a diagonal $2\times 2$-matrix $\mathcal{A}$ where $a_{11}(x)=c_1+\widetilde{a}_{11}(x)$ and $a_{22}(x)=c_2+\widetilde{a}_{22}(x)$ are suitable real perturbations of constant coefficients $c_1>0$ and $c_2\neq 0$. The Cauchy problem \eqref{mainprobIntro} has been studied in \cite{FT,KPRV} for space-dependent variable coefficients using the machinery of pseudodifferential calculus, and more recently in \cite{IPT}, where well-posedness in Sobolev spaces is established through a different approach. The methods involved heavily depend on the regularity of the coefficients, mainly continuity in $t$ and smoothness with respect to the space variable $x$, and limit the physical applications that might involve discontinuous objects, as jump functions, Dirac deltas, etc.

The aim of this paper is to drop the regularity assumptions required in \cite{KPRV, FT} on the coefficients of the operator $P$ generating the Cauchy problem \eqref{mainprobIntro}, as well as those imposed on the Cauchy data, and to establish well-posedness of the resulting problem in the very weak sense. The notion of \textit{very weak solution} has been introduced in the context of hyperbolic equations in \cite{GR} to deal with discontinuous coefficients and employed in related work on hyperbolic equations and systems with multiplicities \cite{DGL, DGL2, G, GB, MRT19}. The same ideas have been recently applied in the context of Schr\"odinger type equations with singular coefficients in \cite{ARST, AACG1, AACG, RST, RY}. In a nutshell, since the presence of distributional coefficients might not allow a meaningful definition of the operator $P$, our approach is to replace $P$ with a family of regularised operators $(P_\eps)_\eps$ where $\eps\to 0$. This is done via convolution with mollifiers of the type $\varphi_{\omega(\eps)}(x)={\omega(\eps)}^{-n}\varphi(x/\omega(\eps))$ where $\omega$ is a regularising net tending to $0$. Our aim is to identify a suitable regularisation method (mollifier and scale) such that the regularised Cauchy problem 
\[
\begin{cases}
    P_\eps u=f_\eps \quad \text{in} \ (0,T] \times \R^n, \\
    u(0,\cdot)=u_{0,\eps}\quad \text{in} \ \R^n,
\end{cases}
\]
is well-posed and the solution $(u_\eps)_\eps$ fulfills \textit{moderate} estimates (i.e. $O(\eps^{-N})$ for some $N\in\N_0$) with respect to the parameter $\eps>0$. This means to prove that a very weak solution $(u_\eps)_\eps$ exists. Its uniqueness is proven modulo negligible nets, i.e., negligible perturbations (i.e. $O(\eps^{q})$ for all $q\in\N_0$) of  coefficients and initial data lead to negligible perturbation of the very weak solution.

Roughly speaking, our main result can be described as follows. We consider
Schr\"odinger operators of the form \eqref{eq.P} whose principal part $\sum_{i,j=1}^nD_{x_i}(a_{ij}(x)D_{x_j})$ is a non-smooth perturbation of
a constant-coefficient non-degenerate operator. More precisely, the
coefficients of the principal part are of the form
\[
    a_{ij}(x)=c_{ij}+\widetilde a_{ij}(x),
\]
where the constant matrix $(c_{ij})_{i,j}$ is real, symmetric and
non-degenerate, while the perturbations $\widetilde a_{ij} \in W^{1,\infty}(\R^n)$, sufficiently small and suitably decaying at infinity.
At the same time, the lower-order coefficients are allowed to be
considerably more singular. For Cauchy data $u_0\in H^{-\infty}(\mathbb{R}^n)$, $ f\in C([0,T];H^{-\infty}(\mathbb{R}^n))$,
we prove that, after a suitable regularisation of the coefficients and the
data, the corresponding family of smooth Cauchy problems admits solutions
$(u_\varepsilon)_\varepsilon$ which form a $H^\infty$-moderate net (see Definition \ref{def.moderate}).
Hence, the original singular problem \eqref{mainprobIntro} admits a $H^\infty$-very weak
solution (see Definition \ref{def.veryweak}). Moreover, this solution is unique modulo $H^\infty$-negligible
nets (see Definition \ref{def.negHs}), in the sense that negligible perturbations of the regularised
coefficients and of the Cauchy data produce a negligible perturbation of
the corresponding solution. Finally, our notion of solution is consistent with the classical theory:
when the coefficients and the data are sufficiently regular so that the
classical Cauchy problem is well-posed, the corresponding very weak
solution converges, as $\varepsilon\to0$, to the classical solution.

The paper is structured as follows. Section 2 collects some preliminary notions on pseudodifferential operators, Sobolev spaces, Sobolev mapping properties and regularisation via convolution with a mollifier. Since Sobolev spaces provide the natural framework for studying well-posedness, we will work with Sobolev-moderate and negligible nets. Section 3 is devoted to ultrahyperbolic operators with singular coefficients, in particular the regularisation of the principal part 
\[
A=\sum_{i,j=1}^nD_{x_i}(a_{ij}(x)D_{x_j}),
\]
where $a_{ij}(x)=c_{ij}+\tilde{a}_{ij}(x)$, $c_{ij} \in \R$ and $\tilde{a}_{ij} \in W^{1,\infty}(\R^n)$ is real valued, for all $i,j=1,\dots,n$. We identify a set of hypotheses on the regularised operator $A_\eps$ that allow us to generalise the Doi's lemma for ultrahyperbolic operators with regular coefficients to this singular context. This is a fundamental step in the proof of our well-posedness result. The full statement of the problem (involving lower order terms as well) and the proof of existence and uniqueness of a $H^\infty$-very weak solution are the contents of Sections 4 and 5. Our method of proof makes use of smoothing estimates as in \cite{FR,FS,FT} combined with careful approximation techniques that clearly state the dependence on the regularising scale. 
We conclude the paper with Section 6, where we prove consistency with the
classical theory, namely that when the equation coefficients are smooth,
every very weak solution converges to the classical solution as
$\varepsilon\to0$.

\section{Preliminaries}
\textbf{Notation.} We make use of the notations $D_t=-i\partial_t$ and $D_{x_j}=-i\partial_{x_j}$. We write $\langle \cdot \rangle=(1+|\cdot|^2)^{1/2}$. By $\mathbb{N}_0$ we mean $\mathbb{N}\cup \lbrace 0 \rbrace$. We denote by $(\cdot,\cdot)_0$ the $L^2$ product. For $s \in \mathbb{R}$ we denote by $H^s(\R^n)$ the Sobolev space of order $s$. With $\langle \cdot, \cdot \rangle$ we denote the standard Euclidean product. The space $C^\infty_b(\R^n)$ is the space of bounded $C^\infty$ functions on $\R^n$, with bounded derivatives at any order. 

\subsection{Basics of pseudodifferential calculus and Sobolev spaces}\label{subsec.pseudo}
In this subsection we recall the basics of pseudodifferential calculus we need throughout the work. For the sake of completeness we start with the notion of (standard) symbol in $\mathbb{R}^{n}\times \mathbb{R}^n$.

Let $a \in C^\infty(\R^{n}\times \R^n)$ and $m \in \mathbb{R}$. We say that \textit{$a$ is a symbol of order $m$}, and we write $a \in S^m(\R^{n}\times \R^n)$ (or simply $a \in S^m$), 
if for all $\alpha,\beta \in \mathbb{N}_0^{n}$  there exists $C_{\alpha,\beta}>0$ such that
\begin{equation}
\label{eq.symbol}
|{\partial_x^\beta\partial_\xi^{\alpha}a(x,\xi)}|\leq C_{\alpha\beta}\langle\xi\rangle^{m-|\alpha|}, \quad (x,\xi) \in \R^n \times \R^n.
\end{equation}

The set of symbols of order $m$ is a Fr\'echet space where, if $a \in S^m$, the semi-norms are given by
\begin{equation*}
\abs{a}_k^{(m)}:=\max_{\abs{\alpha}+\abs{\beta}\leq k}\sup_{(x,\xi) \in \R^n \times \R^n}\abs{\partial_x^\beta\partial_\xi^\alpha a(x,\xi) }\langle \xi \rangle^{-(m-\abs{\alpha})}, \quad k \in \mathbb{N}_0.
\end{equation*}
Equivalently, note that the Fr\'echet topology can be given by using the seminorms 

\begin{equation*}
|a|^{(m)}_{\alpha,\beta}:=\sup_{(x,\xi)\in \mathbb{R}^n\times \mathbb{R}^n}|\partial_\xi^{\alpha}\partial_x^\beta a(x,\xi)|\langle{\xi}\rangle^{-(m-|\alpha|)}, \quad \alpha,\beta \in \mathbb{N}_0^n.
\end{equation*}

If $a \in S^m$, we can associate
to it a \textit{pseudodifferential operator}
using the quantisation formula
\begin{equation*}
\mathrm{Op}(a)u(x):= (2\pi)^{-n}\int_{\mathbb{R}^{2n}}e^{i\langle x-y, \xi \rangle} a(x,\xi)u(y)dyd\xi, \quad u \in \mathscr{S}(\mathbb{R}^n).
\end{equation*}

It is easy to verify that $\mathrm{Op}(a)$ 
is a linear operator that acts continuously on 
$\mathscr{S}(\mathbb{R}^n)$ and extends by duality
to a linear continuous operator on $\mathscr{S}'(\mathbb{R}^n)$. If $A=\mathrm{Op}(a)$, for some $a \in S^m$, we write $A \in \Psi^m(\R^n \times \R^n)$ (or simply $A \in \Psi^m$). 

Furthermore, recall that, if $a \in C^\infty (\R^n \times \R^n)$ is a real-valued smooth function, $H_a$ denotes the \textit{Hamilton vector field} associated with it, which, in standard symplectic coordinates $(x,\xi) \in \R^n\times \R^n$, can be defined as 
\begin{equation}\label{eq.defHa}
H_a:=\sum_{j=1}^n \Bigl( (\partial_{\xi_j}a)\partial_{x_j}-(\partial_{x_j}a)\partial_{\xi_j} \Bigr)
\end{equation}
and then (with $b \in C^\infty (\mathbb{R}^n \times \mathbb{R}^n)$)
\begin{equation}\label{eq.defPoisson}
    \lbrace a, b \rbrace(x,\xi):=H_ab(x,\xi)=\sum_{j=1}^n \Bigl( \partial_{\xi_j}a\partial_{x_j}b-\partial_{x_j}a\partial_{\xi_j}b \Bigr)(x,\xi), \quad (x,\xi)\in \mathbb{R}^n \times \mathbb{R}^n,
\end{equation}
where $\lbrace \cdot, \cdot \rbrace$ are known as \textit{the Poisson brackets}. 

Our well-posedness result will be formulated in the context of \textit{Sobolev spaces}, defined through $\Lambda^s=\mathrm{Op}(\langle \xi \rangle^s)$, as
\[
H^s(\R^n):=\lbrace u \in\mathscr{S}'(\R^n); \ \Lambda^su \in L^2(\R^n) \rbrace, \quad s \in \R,
\]
and endowed with the inner product and the norm given respectively by
\begin{equation*}
(u,v)_{s}:=(\Lambda^s u, \Lambda^s v)_0, \ \ \|u\|_{s}:=(u,u)^{1/2}_{s}, \quad u,v \in H^s(\R^n).
\end{equation*}
Recall also that the space $H^\infty(\R^n)$ and $H^{-\infty}(\R^n)$ are defined by 
\begin{equation}\label{eq.defHinfty-infty}
H^\infty(\R^n):=\bigcap_{s \in \R}H^s(\R^n), \quad H^{-\infty}(\R^n):=\bigcup_{s \in \R} H^s(\R^n),
\end{equation}
respectively. We conclude this subsection by summarising the following fundamental results of the standard pseudodifferential calculus (cf. Section 2.1 in \cite{KPRV}) that we will use in the next sections.
\begin{Th}\label{theo.bound}
Let $a \in S^m$, with $m \in \R$, and let $s \in \R$. Then $\mathrm{Op}(a)$ extends to a bounded linear operator from $H^{s+m}(\R^n)$ to $H^s(\R^n)$ and 
there exist $k=k(n,m,s) \in \mathbb{N}_0$ and $C=C(n,m,s)>0$ such that
\begin{equation*}
    \|\mathrm{Op}(a)u\|_s \leq C |a|_k^{(m)}\|u\|_{s+m}, \quad u \in H^{s+m}(\R^n).
\end{equation*}
\end{Th}
\begin{Th}\label{theo.comp}
Let $a \in S^{m_1}$, $b \in S^{m_2}$, with $m_1,m_2 \in \R$. 
Then
\begin{equation*}
    \mathrm{Op}(a)\mathrm{Op}(b)=\mathrm{Op}(ab)+\mathrm{Op}(r_{m_1+m_2-1}),
\end{equation*}
where $r_{m_1+m_2-1} \in S^{m_1+m_2-1}$ and, for each $k\in \mathbb{N}_0$, there exists $k_1\in \mathbb{N}_0$ and $c_1>0$, depending on $m_1,m_2,k$, such that
\begin{equation*}
    |r_{m_1+m_2-1}|^{(m_1+m_2-1)}_{k}\leq c_1 \, |a|_{k_1}^{(m_1)}|b|_{k_1}^{(m_2)}.
\end{equation*}
In particular,
\begin{equation*}
    [\mathrm{Op}(a),\mathrm{Op}(b)]=\mathrm{Op}(i^{-1}\lbrace a, b \rbrace)+\mathrm{Op}(r_{m_1+m_2-2}),
\end{equation*}
with $r_{m_1+m_2-2} \in S^{m_1+m_2-2}$ and, for each $k\in \mathbb{N}_0$, there exists $k_2\in \mathbb{N}_0$ and $c_2>0$, depending once again on $m_1,m_2,k$, such that
\begin{equation*}
    |r_{m_1+m_2-2}|^{(m_1+m_2-2)}_{k}\leq c_2 \, |a|_{k_2}^{(m_1)}|b|_{k_2}^{(m_2)}.
\end{equation*}
\begin{Th}\label{theo.adjoint}
    Let $a \in S^m$, with $m \in \R$. Then 
    \[
    \mathrm{Op}(a)^\ast=\mathrm{Op}(\bar{a})+\mathrm{Op}(r_{m-1}),
    \]
    with $r_{m-1} \in S^{m-1},$
    and for each $k \in \mathbb{N}_0$ there exists $k' \in \mathbb{N}_0$ and $c'>0$, depending on $m$ and $k$, such that 
    \[
    |r_{m-1}|^{(m-1)}_{k}\leq c'|a|^{(m)}_{k'}.
    \]
\end{Th}

\end{Th}
\begin{Th}\label{theo.ShGa}
    Let $a \in S^1$ be such that there exists $R>0$ so that
    \[
    \mathrm{Re}(a)(x,\xi) \geq 0, \quad \text{for} \ \ |\xi|\geq R.
    \]
    Then, there exist $k=k(n) \in \mathbb{N}_0$ and $C=C(n,R)>0$ such that
    \begin{equation*}
        \mathrm{Re} \, (\mathrm{Op}(a)u,u)_0 \geq -C |a|^{(1)}_k\|u\|_0^2, \quad u \in \mathscr{S}(\R^n).
    \end{equation*}
\end{Th}

\subsection{Regularisations via convolution with a mollifier}\label{subsec.reg}
In this subsection we establish several useful results concerning the regularisation of both the Cauchy data and the coefficients by means of a mollifier. We also specify the mollifier employed throughout the paper and the scale, together with their basic properties.

Let $\varphi \in \mathscr{S}(\R^n)$ be such that $\varphi \geq 0$ and $\int_{\R^n}\varphi(x)dx=1$. Furthermore, let $\omega=\omega(\varepsilon)$, with $\varepsilon \in (0,1]$, be a positive scale, i.e. $\omega$ is a positive and bounded function, such that $\omega(\varepsilon) \rightarrow 0$ as $\varepsilon \rightarrow 0^+$ and $\omega(\varepsilon)\ge c_k\varepsilon^{k}$ for some $k\in\mathbb{N}_0$ and $c_k>0$. It is clearly not restrictive to assume that $\omega(\varepsilon)<1$. Accordingly, we set 
\begin{equation}\label{eq.defphiepsilon}
\varphi_{\omega(\varepsilon)}(x):=\frac{1}{(\omega(\varepsilon))^n} \varphi\Bigl(\frac{x}{\omega(\varepsilon)}\Bigr), \quad x \in \R^n, \ \ \varepsilon \in (0,1].
\end{equation}

In this work we will assume (together with other structural assumptions) that the coefficients of the principal part $A$ of the operator $P$ defined in \eqref{eq.P} belong to the Sobolev space 
\begin{equation*}
W^{1,\infty}(\R^n):=\lbrace u \in L^\infty(\R^n); \quad \partial_{x_k} u \in L^\infty(\R^n), \ \forall k=1,\dots,n \rbrace,
\end{equation*}
endowed with the norm 
\begin{equation*}   \|u\|_{W^{1,\infty}}:=\max_{|\alpha|\leq1}\| \partial^\alpha_{x}u \|_{L^\infty}.
\end{equation*}

The following result will be useful for studying the non-smooth coefficients of our operator \eqref{eq.P}.

\begin{Prop}\label{Prop.v}
Let $v\in L^\infty(\R^n)$ and $\varphi_{\omega(\eps)}$ defined as in \eqref{eq.defphiepsilon}. Hence, for all $\beta\in\N_0^n$ there exists $C_\beta=C_\beta(\|v\|_{L^\infty})>0$ such that 
\begin{equation}\label{eq.v}
|\partial^\beta_x(v\ast\varphi_{\omega(\eps)})(x)|\le C_\beta\omega(\eps)^{-|\beta|},
\end{equation}
holds uniformly in $\eps\in(0,1]$ and $x\in\R^n$. 

\noindent In particular, if $w\in W^{1,\infty}(\mathbb{R}^n)$, then
\begin{equation}\label{eq.regW10}
    |(w\ast\varphi_{\omega(\varepsilon)})(x)|
    \leq C_2,
    \qquad x\in\mathbb{R}^n,
\end{equation}
uniformly in $\varepsilon\in(0,1]$, for some
$C_2=C_2(\|w\|_{W^{1,\infty}})>0$. Moreover, for every
$\beta\in\mathbb{N}_0^n$ with $|\beta|\geq 1$, there exists
$C_\beta=C_\beta(\|w\|_{W^{1,\infty}})>0$ such that
\begin{equation}\label{eq.regW1}
    |\partial_x^\beta(w\ast\varphi_{\omega(\varepsilon)})(x)|
    \leq
    C_\beta\,\omega(\varepsilon)^{-|\beta|+1},
    \qquad x\in\mathbb{R}^n,
\end{equation}
uniformly in $\varepsilon\in(0,1]$.
\end{Prop}
\begin{proof}
To prove \eqref{eq.v} it is sufficient to note that 
\begin{equation*}
    \partial_x^\beta (v \ast \varphi_{\omega(\eps)})(x)= (v \ast \partial_y^{\beta} \varphi_{\omega(\eps)})(x), \quad x \in \R^n,
\end{equation*}
and then, by Young's convolution inequality,
\begin{equation*}
    \|\partial^\beta_x(v \ast \varphi_{\omega(\eps)})\|_{L^\infty} \leq \omega(\eps)^{-|\beta|}\|v\|_{L^\infty}\|\partial_y^\beta \varphi\|_{L^1}, \quad \forall \eps \in (0,1].
\end{equation*}
If $w\in W^{1,\infty}(\mathbb{R}^n)$, then Young's inequality gives
\[
\|w\ast\varphi_{\omega(\varepsilon)}\|_{L^\infty}
\leq
\|w\|_{L^\infty}
\|\varphi_{\omega(\varepsilon)}\|_{L^1}
=
\|w\|_{L^\infty}\|\varphi\|_{L^1},
\]
which proves \eqref{eq.regW10}.

Let now $\beta\in\mathbb{N}_0^n$ with $|\beta|\geq 1$. We can write
$\beta=\beta'+\beta''$, where $|\beta'|=1$ and
$|\beta''|=|\beta|-1$. Since $\partial^{\beta'}w\in L^\infty(\mathbb{R}^n)$,
we have
\[
\partial_x^\beta(w\ast\varphi_{\omega(\varepsilon)})(x)
=
\bigl(
\partial^{\beta'}w
\ast
\partial^{\beta''}\varphi_{\omega(\varepsilon)}
\bigr)(x).
\]
Therefore, again by Young's convolution inequality,
\[
\begin{aligned}
\|\partial_x^\beta(w\ast\varphi_{\omega(\varepsilon)})\|_{L^\infty}
&\leq
\|\partial^{\beta'}w\|_{L^\infty}
\|\partial^{\beta''}\varphi_{\omega(\varepsilon)}\|_{L^1}\\
&\leq
\omega(\varepsilon)^{-|\beta|+1}
\|w\|_{W^{1,\infty}}
\|\partial^{\beta''}\varphi\|_{L^1},
\end{aligned}
\]
which proves \eqref{eq.regW1}.
\end{proof}

Furthermore, the following regularisation argument will play a key role in establishing consistency with the classical theory.

\begin{Prop}\label{prop.conv}
    Let $ v \in H^s(\R^n)$, with $s \in \R$, and let $\varphi_{\omega(\eps)}$ be as in \eqref{eq.defphiepsilon}. Then, for all $\ell \in \mathbb{N}_0$ there exists a constant $C=C(\ell,\varphi)>0$ such that the estimate
    \begin{equation*}
        \|v \ast \varphi_{\omega(\eps)}\|_{s+\ell} \leq C \omega(\eps)^{-\ell}\|v\|_{s},
    \end{equation*}
   holds uniformly in $\eps$.
\end{Prop}
\begin{proof}
As before, we denote by $C$ a positive constant, possibly changing from line to line, independent of $\eps$.
Since $\Lambda^{s+\ell}=\mathscr{F}_{\xi \rightarrow x}^{-1} \langle \xi \rangle^{s+\ell} \mathscr{F}_{y \rightarrow \xi}$, by Plancherel's theorem, we have
\begin{equation*}
        \|v \ast \varphi_{\omega(\eps)}\|_{s+\ell}^2 \leq C \|\langle \xi \rangle^{s+\ell} \widehat{v \ast \varphi_{\omega(\eps)}} \|_0^2 \leq C \|(\langle \xi \rangle^\ell \widehat{\varphi_{\omega(\eps)}})(\langle \xi \rangle^s \hat{v})\|_0^2,
    \end{equation*}
    and then, by H\"older's inequality (applied with $p=1$ and $q=+\infty$), 
    \begin{equation*}
        \|v \ast \varphi_{\omega(\eps)} \|_{s+\ell}^2\leq C \|\langle \xi \rangle^{2\ell}\widehat {\varphi_{\omega(\eps)}}^2\|_{L^\infty} \|v\|_s^2.
    \end{equation*}
    Therefore, since $\varphi \in \mathscr{S}(\R^n)$ and $\widehat{\varphi_{\omega(\eps)}}(\xi)=\hat\varphi(\omega(\eps) \xi)$ we obtain 
    \begin{equation*}
        \|v \ast \varphi_{\omega(\eps)} \|_{s+\ell}\leq C \omega(\varepsilon)^{-\ell} \|v\|_s.
    \end{equation*}
\end{proof}

\subsection{Sobolev moderate and negligible nets}\label{subsec.modnet}
We recall here the notions of $H^s$- and $H^\infty$-moderate nets of functions (cf. \cite{GR}, \cite{AACG},\cite{AACG1}) that, in the next sections, will lead to the notion of $H^s$- and $H^\infty$-very weak solution of the problem \eqref{mainprobIntro}. Furthermore, we also recall the notion of negligible net, that will be useful to treat uniqueness in the context of very weak solutions. 
In what follows, we consider $T>0$ and $s \in \mathbb{R}$ as fixed.
\begin{Def}\label{def.moderate}
    Let $(v_\varepsilon)_\varepsilon \in \lbrace C([0,T];H^s(\mathbb{R}^n) \rbrace^{(0,1]}$, with $s \in \R$. We say that the net $(v_\varepsilon)_\varepsilon$ is $H^s$-moderate if there exist $N \in \mathbb{N}_0$ and $C>0$ such that 
    \[
    \|v_\varepsilon(t,\cdot) \|_{H^s} \leq C \varepsilon^{-N}, \quad \forall t \in [0,T], \ \ \forall \varepsilon \in (0,\eps_0],
    \]
    for some $\eps_0>0$. 
Moreover, if $(v_\varepsilon)_\varepsilon \in \lbrace C([0,T];H^\infty(\mathbb{R}^n) \rbrace^{(0,1]}$ we say that $(v_\varepsilon)_\varepsilon$ is $H^\infty$-moderate if for each $s \in \R$ there exist $N \in \mathbb{N}_0$ and $C>0$ such that 
    \[
    \|v_\varepsilon(t,\cdot) \|_{H^s} \leq C \varepsilon^{-N}, \quad \forall t \in [0,T], \ \ \forall \varepsilon \in (0,\eps_0],
    \]
    for some $\eps_0>0$.
\end{Def}

This notion leads to the notion of $H^s$-moderate and $H^\infty$-moderate regularisations of a distribution. 
\begin{Def}
    Let $v \in C([0,T];\mathscr{D}'( \mathbb{R}^n))$. Moreover, for $\varepsilon \in (0,1]$ and $t \in [0,T]$, define $v_\varepsilon(t,x)=(v(t,\cdot) \ast \varphi_{\varepsilon})(x)$, with $\varphi \in \mathscr{S}(\R^n)$ satisfying  $\int_{\R^n}\varphi(x)dx=1$, and $\varphi_\eps=\eps^{-n}\varphi(\cdot/\eps)$. If $(v_\varepsilon)_\varepsilon$ is $H^s$-moderate (resp. $H^\infty$-moderate), we say that $(v_\varepsilon)_\varepsilon$ is a $H^s$-moderate regularisation (resp. $H^\infty$-moderate regularisation) of $v$.
\end{Def}

\begin{Rem}\label{rmk.h-infty}
     Note that, if $v \in H^{-\infty}(\R^n)$ (recall \eqref{eq.defHinfty-infty}) by repeating essentially the same proof of Proposition \ref{prop.conv}, we get that the net $(v_\eps)_\eps$, defined by $v_\eps=v \ast \varphi_\eps$, is a $H^\infty$-moderate regularisation of $v$.
\end{Rem}

We conclude this section with the definition of negligible nets.
\begin{Def}\label{def.negHs}
Let $(v_\varepsilon)_\varepsilon \in \lbrace C([0,T];H^s(\mathbb{R}^n) \rbrace^{(0,1]}$, with $s \in \R$. We say that $(v_\varepsilon)_\varepsilon$ is $H^s$-negligible if for  any $q \in \mathbb{N}_0$ there exists $C>0$ such that 
    \[
    \|v_\varepsilon(t,\cdot) \|_{H^s} \leq C \varepsilon^{q}, \quad \forall t \in [0,T], \ \ \forall \varepsilon \in (0,\eps_0],
    \]
    for some $\eps_0>0$.
    Moreover, if $(v_\varepsilon)_\varepsilon \in \lbrace C([0,T];H^{\infty}(\mathbb{R}^n) \rbrace^{(0,1]}$, we say that the net $(v_\varepsilon)_\varepsilon$ is $H^\infty$-negligible if for any $s \in \R$ and any $q \in \mathbb{N}_0$ there exists $C>0$ such that 
    \[
    \|v_\varepsilon(t,\cdot) \|_{H^s} \leq C \varepsilon^{q}, \quad \forall t \in [0,T], \ \ \forall \varepsilon \in (0,\eps_0],
    \]
    for some $\eps_0>0$.
\end{Def}

\section{Ultrahyperbolic operators with singular coefficients}
The goal of this section is to  study the main properties of the singular version of ultrahyperbolic operators introduced in \cite{KPRV} (cf. \cite{FT}), namely the higher-order part of the operator $P$ defined in \eqref{eq.P}. We first introduce the net of regularised operators associated with such a singular operator and state the main assumptions. We then prove a new version of Doi's lemma (see Lemma \ref{lemma.Doiepsilon} below) in our setting, which relies on the construction of a suitable net of functions $(q_\eps)_\eps$, provided by Proposition \ref{prop.qepsilon}.

\subsection{Regularisation of singular ultrahyperbolic operators}
In this subsection we consider a class of operators with non-smooth coefficients that can be formally written as
\begin{equation}\label{eq.A}
A=\sum_{i,j=1}^n D_{x_i}(a_{ij}(x)D_{x_j}),
\end{equation}
where the coefficients $a_{ij}$ are given by
\begin{equation}\label{eq.coeffA}
a_{ij}(x)=c_{ij}+\tilde{a}_{ij}(x), \quad i,j=1,\dots,n, \quad  x \in \R^n,
\end{equation}
with $c_{ij} \in \R$ and $\tilde{a}_{ij} \in W^{1,\infty}(\R^n)$ real valued, for all $i,j=1,\dots,n$, and satisfying suitable conditions. In what follows, we also denote by $\mathcal{A}(x)=(a_{ij}(x))_{i,j=1,\dots,n}$ the coefficients matrix.

\begin{Rem}
Recall that every function $u\in W^{1,\infty}(\mathbb{R}^n)$ admits a
globally Lipschitz representative $\widetilde u$, satisfying (cf. \cite{EG}, Section 4.2.3) 
\[
|\widetilde u(x)-\widetilde u(y)|
\leq
\|\nabla u\|_{L^\infty(\mathbb{R}^n)} |x-y|,
\qquad x,y\in\mathbb{R}^n.
\]
Conversely, by Rademacher's theorem, every Lipschitz function on
$\mathbb{R}^n$ is differentiable almost everywhere, and therefore every bounded
Lipschitz function belongs to $W^{1,\infty}(\mathbb{R}^n)$.
Thus, up to the choice of a representative,
$W^{1,\infty}(\mathbb{R}^n)$ can be identified with the space of bounded
Lipschitz functions.

Therefore, one of the interests of our theory is that it allows us to
establish well-posedness results for Cauchy problems induced by variable-coefficient operators whose
principal part is obtained as a bounded Lipschitz perturbation of a
constant-coefficient operator and, as we will see below, with much more
singular lower-order terms.
\end{Rem}

Our purpose is now to construct a net of operators $(A_\eps)_\eps$ associated with $A$ obtained via a regularisation argument. In detail, we will convolve with a mollifier $\varphi_{\omega(\eps)}$ defined (as in \eqref{eq.defphiepsilon}) by
\[
\varphi_{\omega(\varepsilon)}(x):=\frac{1}{(\omega(\varepsilon))^n} \varphi\Bigl(\frac{x}{\omega(\varepsilon)}\Bigr), \quad x \in \R^n, \ \ \varepsilon \in (0,1],
\]
where, recall $\varphi \in \mathscr{S}(\R^n)$, $\varphi \geq 0$, $\int\varphi=1$ and $\omega$ is a positive scale, as specified in Subsection \ref{subsec.reg}.

For all $i,j=1,\dots,n$, we define \textit{the $\varepsilon$-regularised coefficients} of $A$, defined in \eqref{eq.A}, as
\begin{equation}\label{eq.regcoeff}
a_{ij,\varepsilon}(x):=(a_{ij} \ast \varphi_{\omega(\varepsilon)})(x), \quad x \in \R^n, \ \ i,j=1,\dots,n, \ \ \varepsilon \in (0,1].
\end{equation}
We thus define, for $\varepsilon \in (0,1]$, \textit{the $\varepsilon$-regularised operator} 
\begin{equation}\label{eq.Aepsilon}
A_\varepsilon:=\sum_{i,j=1}^nD_{x_i}(a_{ij,\varepsilon}(x)D_{x_j}),
\end{equation}
and denote by $\mathcal{A}_\varepsilon(x)=(a_{ij,\varepsilon}(x))_{i,j=1}^n$ the matrix of $\varepsilon$-regularised coefficients. Finally, $(A_\varepsilon)_\varepsilon$ is \textit{the net of regularised operators} associated with $A$.

\begin{Prop}
Let $A_\eps$ be the operator defined in \eqref{eq.Aepsilon}, with $\eps \in (0,1]$. Then, $A_\eps=\mathrm{Op}(a_\eps)$ is an operator with symbol $a_\eps(x,\xi)$ fulfilling the following statement: 
for all $\alpha,\beta \in \mathbb{N}_0^n$ there exists a constant $C_{\alpha,\beta}>0$ such that
\begin{equation}\label{eq.conda}
|\partial_x^\beta\partial_\xi^\alpha a_{\varepsilon}(x,\xi)|\leq C_{\alpha,\beta} \, \omega(\varepsilon)^{-|\beta|}\langle \xi \rangle^{2-|\alpha|}, \quad \forall (x,\xi) \in \R^n \times \R^n.
\end{equation}
\end{Prop}
    
\begin{proof}    
Note that $A_\eps=\mathrm{Op}(a_\eps)$, where 
\begin{equation}\label{eq.defaeps}
\begin{split}
a_\varepsilon(x,\xi)&=\sum_{i,j=1}^n a_{ij,\varepsilon}(x)\xi_i\xi_j+\sum_{i,j=1}^n (D_{x_i}a_{ij,\varepsilon})(x)\xi_j,  \\
&=:a_{2,\varepsilon}(x,\xi)+a_{1,\varepsilon}(x,\xi), \quad (x,\xi)\in \R^{n}\times \R^n,
\end{split}
\end{equation}
and $a_{ij,\eps}$ defined as in \eqref{eq.regcoeff}. Then, since $\tilde{a}_{ij} \in W^{1,\infty}(\R^n)$ for all $i,j=1,\dots,n$, by \eqref{eq.regW10} and \eqref{eq.regW1} we obtain that for all $\alpha,\beta \in \mathbb{N}_0^n$ there exists $C_{\alpha,\beta}'>0$ such that  
\begin{equation}\label{eq.conda2}
    |\partial_x^\beta \partial_\xi^\alpha a_{2,\varepsilon}(x,\xi)|
    \leq C'_{\alpha,\beta}
    \omega(\varepsilon)^{-\max\{|\beta|-1,0\}}
    \langle\xi\rangle^{2-|\alpha|},
    \qquad (x,\xi)\in\mathbb{R}^n\times\mathbb{R}^n,
\end{equation}
and, similarly (since $D_{x_i}a_{ij} \in L^\infty(\R^n)$ for all $i,j$) by \eqref{eq.v}, there exists $C_{\alpha,\beta}''>0$ such that 
\begin{equation}\label{eq.conda1}
|\partial_x^\beta\partial_\xi^\alpha a_{1,\varepsilon}(x,\xi)|\leq C_{\alpha,\beta}'' \, \omega(\varepsilon)^{-|\beta|}\langle \xi \rangle^{1-|\alpha|},  \quad \forall (x,\xi) \in \R^n \times \R^n.
\end{equation}
Therefore, \eqref{eq.conda} easily follows from \eqref{eq.defaeps}.
\end{proof}

In order to study the well-posedness (see Subsection \ref{subsec.sop} below) of the Cauchy problem
\eqref{mainprobIntro} associated with the operator $P$, we assume the
hypotheses below on the net of regularised operators
$(A_\varepsilon)_\varepsilon$, representing the regularised principal
part of $P$. These conditions, together with further assumptions on the
regularised lower-order parts of $P$, are satisfied for a class of
(singular) operators that motivated our work (see
Theorem~\ref{theo_with_sc}). We assume the following hypotheses.
\begin{hp}\label{hp1}
For each $\varepsilon \in (0,1]$ the coefficients matrix $\mathcal{A}_\varepsilon$ is real and symmetric.
\end{hp}
\begin{hp}\label{hp2} There exists a universal constant $\mu>0$ such that for any $\varepsilon \in (0,1]$ one has 
\begin{equation*}
\mu^{-1}|\xi| \leq |\mathcal{A}_\varepsilon(x) \cdot \xi| \leq \mu |\xi|, \quad \forall (x,\xi) \in \R^n \times \R^n.
\end{equation*}
\end{hp}
\begin{hp}\label{hp3}
There exist $\nu>0$ sufficiently small and $N \in \mathbb{N}$, $N>1$, such that, for any $\varepsilon \in (0,1]$ one has
\begin{equation*}
|\partial_{x_k} a_{ij,\varepsilon}(x)| \leq \nu \langle x \rangle^{-N}, \quad \forall x \in \R^n, \ \ i,j,k=1,\dots,n. 
\end{equation*}
\
\end{hp}

As anticipated above, let us emphasise that, once these conditions are required on the net $(A_\varepsilon)_\varepsilon$,
the operator $A$ can be interpreted as the
"singular-coefficients" version of the \textit{ultrahyperbolic
operators} studied in \cite{FT} and \cite{KPRV}.
We conclude this subsection with a few comments and motivations related to the hypotheses \ref{hp1}-\ref{hp3}.

In the first place, it is important to note that the ultrahyperbolic condition \ref{hp2} (as named in \cite{KPRV}) is, in fact, a non-degeneracy condition.
Therefore, (uniformly) elliptic variable coefficient operators, satisfy the "ultrahyperbolic condition" 
\begin{equation*}
    C^{-1}|\xi|\leq |\mathcal{A}(x)\cdot \xi | \leq C|\xi|, \quad \forall (x,\xi) \in \R^n \times \R^n.
\end{equation*}
Consequently, our theory includes the non-smooth variable coefficients Schr\"odinger operator of the form $P=D_t-A(x,D_x)+\text{l.o.t.}$ (where l.o.t. means lower order terms), with $A$ non-smooth elliptic operator. Indeed, in the very weak solution
framework, the non-smooth coefficients of $A$ are replaced by suitable
smooth regularisations. The following lemma shows that, whenever these
regularisations satisfy a uniform ellipticity condition, the corresponding
regularised operators also satisfy the ultrahyperbolic condition.

\begin{Lemma}
Let us consider $a \in S^2$ defined by
\begin{equation}\label{eq.aquad}
a(x,\xi)=\langle \mathcal{A}(x)\cdot \xi,\xi \rangle, \quad (x,\xi) \in \R^n \times \R^n,
\end{equation}
where $\mathcal{A}(x)=(a_{ij}(x))$ and $a_{ij} \in C^\infty(\R^n)$ real valued with bounded derivatives at any order, for all $i,j=1,\dots,n$. 

If $a(x,\xi)$ satisfies the ellipticity condition 
\begin{equation}\label{eq.ellsymbol}
    C^{-1} |\xi|^2 \leq a(x,\xi) \leq C |\xi|^2, \quad \forall (x,\xi) \in \R^n \times \R^n,
\end{equation}
for some $C>0$, then there exists $C'>0$ such that 
\begin{equation}\label{eq.ultraell}
    {C'}^{-1}|\xi|\leq |\mathcal{A}(x)\cdot \xi | \leq C'|\xi|, \quad \forall (x,\xi) \in \R^n \times \R^n.
\end{equation}    
\end{Lemma}

\begin{proof}
In the first place, under the hypothesis \eqref{eq.ellsymbol}, by the Cauchy--Schwarz inequality (applied to \eqref{eq.aquad}), we get 
\begin{equation*}
     C^{-1}|\xi|^2 \leq a(x,\xi) \leq |\mathcal{A}(x)\cdot \xi | |\xi| , \quad \forall (x,\xi) \in \R^n \times \R^n,
\end{equation*}
which implies 
\begin{equation*}
    |\mathcal{A}(x)\cdot \xi | \geq C^{-1}|\xi|, \quad \forall (x,\xi) \in \R^n \times \R^n.
\end{equation*}
Moreover, by the bounded assumption on the coefficients $a_{ij}(x)$, we have
\begin{equation*}
    |\mathcal{A}(x)\cdot \xi | \leq C'|\xi|, \quad \forall (x,\xi) \in \R^n \times \R^n,
\end{equation*}
for some $C'>0$. Hence, if $a \in S^2$ is of the form \eqref{eq.aquad} and satisfies \eqref{eq.ellsymbol}, we obtain that \eqref{eq.ultraell} holds.
\end{proof}

 \begin{Rem}\label{rem.satshp3}
 As for the hypothesis \ref{hp3}, it is immediate to construct an operator having coefficients for which the net of regularised coefficients satisfy \ref{hp3}. Explicitly, if we consider  $(b_{ij})_{i,j=1}^n$ where $b_{ij} \in W^{1,\infty}(\R^n)$ such that $\|b_{ij}\|_{W^{1,\infty}}\leq \nu$, for $\nu>0$ sufficiently small, setting 
     \begin{equation*}
     \tilde{a}_{ij}(x):=\langle x \rangle^{-N}b_{ij}(x), \quad x \in \R^n, \ \ i,j=1,\dots,n,
     \end{equation*}
     we get that the coefficients $a_{ij,\varepsilon}$ defined as in \eqref{eq.regcoeff} satisfy \ref{hp3}. 

     \noindent Indeed, using the Peetre's inequality $\langle x-y \rangle^{-N} \leq C_N \langle x \rangle^{-N}\langle y \rangle^{N}$, for all $\varepsilon \in (0,1]$ and for all $i,j,k=1,\dots,n$, we have
     \begin{equation}\label{eq.peetre}
     \begin{split}
     |\partial_{x_k}\tilde{a}_{ij,\varepsilon}(x)|&=\Bigl|\int_{\R^n}\partial_{x_k}(\langle x-y \rangle^{-N}b_{ij}(x-y))\varphi_{\omega(\varepsilon)}(y)dy\Bigr|\\
     &\leq \int_{\R^n}|\partial_{x_k}(\langle x-y \rangle^{-N}b_{ij}(x-y))\varphi_{\omega(\varepsilon)}(y)|dy \\
     &\leq C_{N}\langle x \rangle^{-N}\|b_{ij}\|_{W^{1,\infty}}\int_{\R^n}\langle y \rangle^{N}\varphi_{\omega(\varepsilon)}(y)dy \\
     &\leq C_{N,\varphi} \, \nu \langle x \rangle^{-N}, \quad \forall x \in \R^n,
     \end{split}
     \end{equation}
     for some $C_{N,\varphi}>0$, uniformly in $\eps \in (0,1]$.
 \end{Rem}

We finally present some examples of operators that fall under the scope of our study.  
In the next example, we consider models that generalise non-degenerate smooth operators, in the singular setting, defined through a diagonal coefficients matrix (see \eqref{eq.diagmatrix}).
In this context, depending on the value of the constant $c_2$ below, 
we recover an elliptic operator with non-smooth coefficients (if $c_2 > 0$) 
and a "hyperbolic" operator (if $c_2 < 0$).

 \begin{Ex}\label{ex.ultradiag}
Let us consider on $\R^2$ the operator 
     \begin{equation*}
     A=\sum_{i,j=1}^2D_{x_i}(a_{ij}(x)D_{x_j}),
     \end{equation*}
     defined by the matrix
     \begin{equation}\label{eq.diagmatrix}
     \mathcal{A}(x)=
     \begin{pmatrix}
         a_{11}(x) & 0 \\
         0 & a_{22}(x)
     \end{pmatrix}, \quad x \in \R^2,
     \end{equation}
     where
     \begin{equation*}
     \begin{split}
     a_{11}(x)&=c_1+\tilde{a}_{11}(x),\\
     a_{22}(x)&=c_2+\tilde{a}_{22}(x),
     \end{split}
     \end{equation*}
     with $c_1>0$, $c_2\in \R\setminus \lbrace 0 \rbrace$, and $\tilde{a}_{11}(x), \tilde{a}_{22}(x) \in W^{1,\infty}(\R^2)$ real functions, such that 
     \begin{equation*}
     \begin{split}
     & 0\leq |\tilde{a}_{jj}(x)| \leq \frac{|c_j|}{2}, \quad \forall x \in \R^2, \ \ j=1,2, \\
     & |\partial_{x_i}\tilde{a}_{jj}(x)| \leq \nu \langle x \rangle^{-N}, \quad \text{for a.e.} \; x \in \R^2, \ \ i,j=1,2,
     \end{split}
     \end{equation*}
     for some $N>1$ and some $\nu>0$ sufficiently small.

     Under these assumptions, we get that \ref{hp1} is trivially satisfied. As for \ref{hp2}, for $\varepsilon \in (0,1]$, we get 
     \begin{equation*}
     |\mathcal{A}_\varepsilon(x) \cdot \xi |^2=a_{11,\varepsilon}(x)^2\xi_1^2+a_{22,\varepsilon}(x)^2\xi_2^2.
     \end{equation*}
     Hence, for all $(x,\xi) \in  \R^n\times \R^n$,
     \begin{equation*}
     \begin{split}
     &|\mathcal{A}_\varepsilon(x) \cdot \xi| \leq \max{\Bigl\lbrace\frac{3}{2}c_1,\frac{3}{2}|c_2|\Bigr\rbrace}|\xi| \\
     & |\mathcal{A}_\varepsilon(x) \cdot \xi| \geq \min{\Bigl\lbrace\frac{1}{2}c_1,\frac{1}{2}|c_2|\Bigr\rbrace}|\xi|,
     \end{split}
     \end{equation*}
     and \ref{hp2} easily follows. Finally, regarding \ref{hp3}, by reasoning as in \eqref{eq.peetre} we have
     \begin{equation}\label{eq.smallcoeffex}
     \begin{split}
     |\partial_{x_i}a_{jj,\varepsilon}(x)| &\leq \int_{\R^2}|\partial_{x_i}a_{jj}(x-y)||\varphi_{\omega(\varepsilon)}(y)|dy \\
     & \leq \nu \int_{\R^2}\langle x-y \rangle^{-N}|\varphi_{\omega(\varepsilon)}(y)|dy \\
     & \leq C_N \, \nu \int_{\R^2} \langle x \rangle^{-N}\langle y \rangle^N |\varphi_{\omega(\varepsilon)}(y)|dy \\
     & \leq C_{N,\varphi} \, \nu \langle x \rangle^{-N}, \quad \forall x \in \R^2, \ \ i,j=1,2,
     \end{split}
     \end{equation}
     for some $C_{N,\varphi}>0$, uniformly in $\eps \in (0,1]$.
 \end{Ex}

 \subsection{Doi's Lemma for ultrahyperbolic operators with singular coefficients}
  Next, for the operators under consideration the following fundamental Proposition holds (cf. \cite{FT}). 
 \begin{Prop}\label{prop.qepsilon}
     Let $A$ be an operator as in \eqref{eq.A}, with coefficients of the form \eqref{eq.coeffA} and let $(A_\varepsilon)_\varepsilon$ be the corresponding net of regularised operators. If \ref{hp1}, \ref{hp2} and \ref{hp3} hold, then there exists a net of real-valued functions $(q_\varepsilon)_\varepsilon$, with $q_\varepsilon \in C^\infty(\R^n \times \R^n)$ that satisfies the following:
     
     (i) For each $\alpha,\beta \in \mathbb{N}_0^n$ there exists a constant $C_{\alpha,\beta}>0$ such that  
    \begin{equation}\label{prop.q}
    | \partial_\xi^\alpha\partial_x^\beta q_\varepsilon(x,\xi)|\leq 
   \left\{ \begin{array}{ll}
    C_{\alpha,\beta} \langle x\rangle \langle \xi \rangle^{-|\alpha|},     & \text{if}\,\,|\beta|=0,  \\ 
     C_{\alpha,\beta} \omega(\varepsilon)^{-|\beta|+1} \langle \xi \rangle^{-|\alpha|},    &  \text{if}\,\,|\beta|\geq 1 , 
    \end{array}\right.
\end{equation}
for all $(x,\xi) \in \R^n \times \R^n$.

(ii) There exist $C_1, C_2>0$ such that 
\begin{align}
    H_{a_{2,\varepsilon}}q_\varepsilon(x,\xi)\geq C_1|\xi|-C_2,
\end{align}
for all $(x,\xi) \in \R^n \times \R^n$, where $H_{a_2,\eps}$ is the Hamilton vector field defined in \eqref{eq.defHa}.
\end{Prop}
\begin{proof}
For each $\varepsilon \in (0,1]$ let us define (with $\mu>0$ as in \ref{hp2} and $C_1>0$)
\[
q_\varepsilon(x,\xi):=C_1\mu^2\langle \xi \rangle^{-1}\sum_{j=1}^n x_j\partial_{\xi_j}a_{2,\varepsilon}(x,\xi), \quad (x,\xi) \in \R^n \times \R^n,
\]
where recall $a_{2,\eps}$ is defined in \eqref{eq.defaeps}.
With this definition of $q_\varepsilon$, condition $(i)$ is satisfied because of the properties of $a_{2,\varepsilon}$ (see \ref{hp3} and \eqref{eq.conda2}).

Moreover, using the expression \eqref{eq.defPoisson}, we have 
\[
\begin{split}
H_{a_2,\eps}q_\eps(x,\xi)&=
(\langle \nabla_\xi a_{2,\varepsilon},\nabla_x q_\varepsilon \rangle - \langle \nabla_x a_{2,\varepsilon},\nabla_\xi q_\varepsilon \rangle)(x,\xi) \\
&=C_1\mu^2|\nabla_\xi a_{2,\varepsilon}(x,\xi)|^2\langle \xi \rangle ^{-1}+C_1\mu^2\langle \xi \rangle ^{-1}\sum_{j,k=1}^n x_j\partial_{\xi_k}a_{2,\varepsilon}(x,\xi)\partial_{\xi_j}\partial_{x_k}a_{2,\varepsilon}(x,\xi) \, + \\
&\quad \quad -\langle\nabla_x a_{2,\varepsilon}(x,\xi),\nabla_\xi q_{\varepsilon}(x,\xi) \rangle\\
&\geq C_1\mu^2|\nabla_\xi a_{2,\varepsilon}(x,\xi)|^{2}\langle \xi \rangle ^{-1} - C_1\mu^2\langle \xi \rangle ^{-1}\sum_{j,k=1}^n |x_j\partial_{\xi_k}a_{2,\varepsilon}(x,\xi)\partial_{\xi_j}\partial_{x_k}a_{2,\varepsilon}(x,\xi)| \, + \\
&\quad\quad - |\nabla_x a_{2,\varepsilon}(x,\xi)||\nabla_\xi q_{\varepsilon}(x,\xi)|, \quad \forall (x,\xi) \in \R^{2n}.
\end{split}
\]
Hence, since 
\[
|\nabla_\xi a_{2,\varepsilon}(x,\xi)|=2|\mathcal{A}_\varepsilon(x)\cdot \xi|, \quad (x,\xi) \in \R^n \times \R^n,
\]
by \ref{hp2} and \ref{hp3} we obtain
\[
H_{a_{2,\varepsilon}} q_\varepsilon(x,\xi)\geq 2C_1|\xi|^{2}\langle \xi \rangle ^{-1}  -\nu C'\langle \xi \rangle\geq (2C_1-\nu C'')|\xi|-C_2, \quad \forall (x,\xi)\in \R^n\times \R^n,
\]
for some constants $C',C'',C_2$ depending on $\mu,C_1$ and on the (universal) bounds of the coefficients $a_{ij}$, but not on $\varepsilon$. Therefore, for $\nu>0$ sufficiently small
\begin{equation*}
H_{a_{2,\varepsilon}} q_\varepsilon(x,\xi) \geq C_1 |\xi|-C_2\quad \forall(x,\xi)\in\R^{n}\times \R^n,
\end{equation*}
and then $(ii)$. This concludes the proof of the proposition.
\end{proof}
We now prove the following generalised version of Doi's Lemma (cf. \cite{D}, \cite{FT}).
\begin{Lemma}\label{lemma.Doiepsilon}
 Let $A$ be an operator as in \eqref{eq.A} with coefficients of the form \eqref{eq.coeffA}, and let $(A_\varepsilon)_\varepsilon$ be the corresponding net of regularised operators. Moreover, assume that there exists a net of real-valued functions $(q_\varepsilon)_\varepsilon$, with $q_\varepsilon \in C^\infty(\R^n \times \R^n)$ that satisfies the following statements:
     
     (i) For each $\alpha,\beta \in \mathbb{N}_0^n$ there exists a constant $C_{\alpha,\beta}>0$ such that  
    \begin{equation}\label{eq.qip}
    | \partial_\xi^\alpha\partial_x^\beta q_\varepsilon(x,\xi)|\leq 
   \left\{ \begin{array}{ll}
    C_{\alpha,\beta} \langle x\rangle \langle \xi \rangle^{-|\alpha|},     & \text{if}\,\,|\beta|=0,  \\ 
     C_{\alpha,\beta} \omega(\varepsilon)^{-|\beta|+1} \langle \xi \rangle^{-|\alpha|},    &  \text{if}\,\,|\beta|\geq 1 , 
    \end{array}\right.
\end{equation}
for all $(x,\xi) \in \R^n \times \R^n$.

(ii) There exist $C_1, C_2>0$ such that
\begin{equation}\label{eq.qiip}
    H_{a_{2,\varepsilon}}q_\varepsilon(x,\xi)\geq C_1|\xi|-C_2,
\end{equation}
for all $(x,\xi) \in \R^n \times \R^n$.

Then, there exists a net of symbols $(d_\varepsilon)_\varepsilon$ such that 
for all $\alpha,\beta \in \mathbb{N}_0^n$  
\begin{equation}\label{eq.pSM}
    |\partial^\beta_x\partial^\alpha_\xi d_\varepsilon(x,\xi)|
\leq
C_{\alpha,\beta}
\omega(\varepsilon)^{-\max\{|\beta|-1,0\}}
\langle\xi\rangle^{-|\alpha|},
\quad
\forall (x,\xi)\in\mathbb{R}^n\times\mathbb{R}^n, \,
\forall
\varepsilon\in(0,1].
\end{equation}
for some $C_{\alpha,\beta}>0$ and  
\begin{equation}\label{eq.pDoi}
    H_{a_{2,\varepsilon}}d_\varepsilon(x,\xi)\geq \langle x \rangle^{-N}|\xi|-C,\quad \forall (x,\xi)  \in \R^n \times \R^n,\, \forall \varepsilon \in (0,1],
\end{equation}
for some $C>0$ independent of $\varepsilon$ and $N\in \mathbb{N}$, $N>1$.
\end{Lemma}
\begin{proof}
In the first place, we note that, by hypothesis \eqref{eq.qip}, there exists a fixed constant $K>0$ such that, for each $\varepsilon \in (0,1]$ one has
\begin{equation}\label{eq.estqepsilonK}
|q_\varepsilon(x,\xi)| \leq K \langle x \rangle, \quad \forall (x,\xi) \in \R^n \times \R^n.
\end{equation}
Let us define $\lambda(t) := \langle t \rangle^{-N}$ for $t \geq 0$, and extend it to $t < 0$ by setting $\lambda(t) := \lambda(0)$ if $t<0$. By Lemma 3.1 in \cite{D}  there exists a nonnegative real-valued function $f \in C^\infty([0,+\infty))$ such that (cf. \cite{FT})
\begin{equation}\label{eq.fdoi1}
f'(t) \geq \lambda(K^{-1}t - 10), \quad \forall t \geq 0,
\end{equation}
and
\begin{equation}\label{eq.foi2}
|f^{(m)}(t)| \leq C_m \left( \lambda(0) + \int_0^t \lambda(s)ds \right) (1 + t)^{-m}, \quad \forall t \geq 0,
\end{equation}
for some constants $C_m > 0$.
Therefore, by \eqref{eq.estqepsilonK} and \eqref{eq.fdoi1} we have 
\begin{equation}\label{eq.fder1}
f'(|q_\varepsilon(x,\xi)|) \geq \lambda(K^{-1}|q_\varepsilon(x,\xi)|-10) \geq \lambda (\langle x \rangle -10) \geq \lambda(|x|)=\langle x \rangle^{-N}, \quad \forall (x,\xi) \in \R^n\times \R^n.
\end{equation}
Now, we take $\delta>0$ sufficiently small to be fixed later, and $\phi \in C^\infty(\R)$ such that $\phi(t)=0$, if $t\leq 1$, $\phi(t)=1$ if $t \geq 2$, and $\phi'(t) \geq 0$ on $\R$. We set $\phi_+(t):=\phi(t/\delta)$, $\phi_-(t):=\phi(-t/\delta)$ and $\phi_0:=1-\phi_+-\phi_-$. For $\eps \in (0,1]$ we define $\psi_{+,\varepsilon},\psi_{-,\varepsilon},\psi_{0,\varepsilon} \in S^0$ as
\begin{equation*}
\psi_{0,\varepsilon}(x,\xi):=\phi_0(q_\varepsilon(x,\xi)/\langle x \rangle), \quad \psi_{+,\varepsilon}(x,\xi):=\phi_+(q_\varepsilon(x,\xi)/\langle x \rangle), \quad \psi_{-,\varepsilon}(x,\xi):=\phi_{-}(q_\varepsilon(x,\xi)/\langle x \rangle). 
\end{equation*}
By \eqref{eq.fdoi1}, for all $\alpha,\beta \in \mathbb{N}_0^n$ there exists a constant $C_{\alpha,\beta}>0$ such that
\begin{equation*}
|\partial_x^\beta \partial_\xi^{\alpha}f(|q_\varepsilon(x,\xi)|)|\leq C_{\alpha,\beta}\omega(\varepsilon)^{-|\beta|+1}\langle \xi \rangle^{-|\alpha|}
\quad \text{on} \,\,\, \mathrm{supp} \, \psi_{+,\varepsilon} \cup \mathrm{supp} \, \psi_{-,\varepsilon}.
\end{equation*} 

\noindent Next, for $\varepsilon \in (0,1]$, we define $d_\varepsilon \in S^0$ as
\begin{equation}\label{Hp}
d_\varepsilon(x,\xi):=\frac{q_\varepsilon(x,\xi)}{\langle x \rangle}\psi_{0,\varepsilon}(x,\xi)+(f(|q_\varepsilon(x,\xi)|)+2\delta)(\psi_{+,\varepsilon}(x,\xi)-\psi_{-,\varepsilon}(x,\xi)), \quad (x,\xi) \in \R^n \times \R^n.
\end{equation}
By this definition of $d_\eps$ we get that \eqref{eq.pSM} is satisfied due to the property of $q_\eps$. Hence, the goal is to prove \eqref{eq.pDoi}. 

In the first place, we note that 
by hypotheses (i) and (ii) on $(q_\varepsilon)_\varepsilon$, by choosing $\delta>0$ sufficiently small, for all $\varepsilon \in (0,1]$ we have 
\begin{equation*}
H_{a_{2,\varepsilon}}\Bigl(\frac{q_\varepsilon}{\langle x \rangle}\Bigr)=\frac{H_{a_{2,\varepsilon}}q_\varepsilon}{\langle x \rangle}-\frac{q_\varepsilon}{\langle x \rangle}\frac{x \cdot \nabla_\xi a_{2,\varepsilon}}{\langle x \rangle^2}\geq C_1' \frac {|\xi|}{\langle x \rangle}-C_2' \quad \text{on}\,\, \mathrm{supp} \, \psi_{0,\varepsilon},
\end{equation*}
 for some universal constants $C_1',C_2'>0$, independent of $\varepsilon$. 
Hence, by using the properties of $f$ and the inclusions  $\mathrm{supp} \, \phi'_+, \mathrm{supp} \, \phi'_- \subseteq \mathrm{supp}\,\phi_0$, for all $\varepsilon \in (0,1]$ we get
\begin{align}
H_{a_{2,\varepsilon}}d_\varepsilon&=H_{a_{2,\varepsilon}}\Bigl(\frac{q_\varepsilon}{\langle x \rangle}\Bigr)\psi_{0,\varepsilon}+  f'(\abs{q_\varepsilon})(H_{a_{2,\varepsilon}}q_\varepsilon)(\psi_{+,\varepsilon} +\psi_{-,\varepsilon})\\
&\quad +\Bigl(f(\abs{q_\varepsilon})+2\delta-\frac{\abs{q_\varepsilon}}{\langle x \rangle}\Bigr)\Bigl(\phi_+'\Bigl(\frac{q_\varepsilon}{\langle x \rangle}\Bigr)-\phi_-'\Bigl(\frac{q_\varepsilon
    }{\langle x \rangle}\Bigr)\Bigr)H_{a_{2,\varepsilon}}\Bigl(\frac{q_\varepsilon}{\langle x \rangle}\Bigr) \nonumber\\
    &\geq H_{a_{2,\varepsilon}}\Bigl(\frac{q_\varepsilon}{\langle x \rangle}\Bigr)\psi_{0,\varepsilon}+  f'(\abs{q_\varepsilon})(H_{a_{2,\varepsilon}}q_\varepsilon)(\psi_{+,\varepsilon} +\psi_{-,\varepsilon}) -C_3,\label{eq.Hap}
    \end{align}
for some $C_3>0$.

 To estimate the second term on the RHS of \eqref{eq.Hap} we use hypothesis (ii) and \eqref{eq.fder1}. 
 Putting everything together we obtain 
\begin{equation*}
\begin{split}
H_{a_{2,\varepsilon}}d_\varepsilon(x,\xi)&\geq C_4\Big(\langle x \rangle^{-1} \psi_{0,\varepsilon}+\langle x\rangle^{-N}(\psi_{+,\varepsilon}+\psi_{-,\varepsilon})\Big)\abs{\xi}-C_5, \\
& \geq C_4 \langle x \rangle^{-N}(\psi_{0,\varepsilon}+\psi_{+,\varepsilon}+\psi_{-,\varepsilon})\abs{\xi}-C_5\\
& =C_4\langle x \rangle^{-N}|\xi|-C_5, \quad \forall (x,\xi) \in \R^n \times \R^n,
\end{split}
\end{equation*}
for some new constants $C_4,C_5>0$.
In conclusion, possibly rescaling the symbol $d_\eps$, we showed that there exists a constant $C>0$ such that for all $\varepsilon \in (0,1]$
\begin{equation*}
H_{a_{2,\varepsilon}}d_\varepsilon\geq \langle x \rangle^{-N}\abs{\xi}-C, \quad \forall (x, \xi) \in \R^n\times \R^n.
\end{equation*}
\end{proof}
\begin{Rem}
    Let us point out that, as a consequence of Proposition \ref{prop.qepsilon}, we obtain that, if \ref{hp1}, \ref{hp2}, \ref{hp3} hold, there exists a net  $(d_\eps)_\varepsilon$of symbols of order $0$ such that \eqref{eq.pSM} and \eqref{eq.pDoi} are satisfied.
\end{Rem}
\begin{Rem}\label{rmk.C1}
    Note also that, possibly by rescaling the symbol $d_\eps$ with a universal constant (the same for each $\varepsilon \in (0,1]$), instead of \eqref{eq.pDoi} we obtain
\begin{equation}\label{eq.pDoi2}
    H_{a_{2,\varepsilon}}d_\eps \geq C_1 \langle x \rangle^{-N} |\xi| - C', 
    \quad \forall (x, \xi) \in \mathbb{R}^n \times \mathbb{R}^n,
\end{equation}
where $C_1, C' > 0$ are universal constants (independent of $\varepsilon$) and $C_1$ can be taken arbitrarily large.
\end{Rem}
\section{Statement of the problem and regularised version}\label{sec.stateandreg}
The aim of this section is to introduce the proper framework to study the problem \eqref{mainprobIntro}, namely the Schr\"odinger ultrahyperbolic equations induced by an ultrahyperbolic operator  $A$ defined as in \eqref{eq.A}. Subsequently, we study a regularised version of problem \eqref{mainprobIntro} and we prove Theorem \ref{theo.apriori} which will be fundamental for obtaining the existence and the uniqueness results.

\subsection{Statement of the problem}\label{subsec.sop}
We focus on \textit{singular} operators, that can be formally written as 
\begin{equation}\label{def.P}
  P=D_t-\sum_{i,j=1}^nD_{x_i}(a_{ij}(x)D_{x_j})- \sum_{k=1}^nb_k(x)D_{x_k}-V(x),
\end{equation}
where, in general, $a_{ij},b_k,V \in \mathscr{D}'(\R^n)$, for all $i,j,k=1,\dots,n$. 
As discussed in the previous section we assume that the net of regularised operators $(A_\varepsilon)_\varepsilon$, associated with the principal part of $P$, defined by 
\[
A_\varepsilon=\sum_{i,j=1}^nD_{x_i}(a_{ij,\varepsilon}(x)D_{x_j}), \quad \varepsilon \in (0,1], 
\]
with $a_{ij,\eps}$ as in \eqref{eq.regcoeff}, satisfies \ref{hp1}, \ref{hp2} and \ref{hp3}.
Furthermore, denoting by $(\varphi_{\omega(\varepsilon)})_{\varepsilon \in (0,1]}$ the family of mollifiers defined as in \eqref{eq.defphiepsilon}, we define 
\[
b_{k,\eps}:=b_k \ast \varphi_{\omega(\eps)}, \ \ k=1,\dots,n, \quad \quad V_\eps:=V \ast \varphi_{\omega(\eps)},
\]
and we work under the following additional hypotheses on the "first-order" and "zero-order" parts of the operator $P$. We assume:
\begin{hp}\label{hp4}
There exists a universal constant $c_0>0$ such that, for each $\varepsilon \in (0,1]$ (with $N>1$ as in \ref{hp3} and $b_{k,\varepsilon}=b_k \ast \varphi_{\omega(\varepsilon)}$),
\begin{equation}
|\mathrm{Im}(b_{k,\varepsilon})(x)| \leq c_0 \langle x \rangle^{-N}, \quad \forall x \in \R^n, \ \forall k = 1, \dots, n,
\end{equation}
and there exists $N_1 \in \mathbb{N}$, such that for each $\beta \in \mathbb{N}_0^n$ and for each $\eps \in (0,1]$
\begin{equation}
\label{eq.bkeps}
    |\partial_x^\beta b_{k,\eps}(x)|\leq C_\beta\omega(\eps)^{-|\beta|-N_1}, \quad \forall x \in \R^n, \ \forall k = 1, \dots, n,
\end{equation}
for some constant $C_\beta>0$.
\end{hp}
\begin{hp}\label{hp5}
There exists $N_2 \in \mathbb{N}$, such that for each $\gamma \in \mathbb{N}_0^n$ and for each $\varepsilon \in (0,1]$ (with $V_\varepsilon=V \ast \varphi_{\omega(\varepsilon)}$),
\begin{equation}\label{eq.Veps}
|\partial_x^\gamma V_\varepsilon(x)| \leq C_\gamma \omega(\eps)^{-|\gamma|-N_2}, \quad \forall x \in \R^n,
\end{equation}
for some constant $C_\gamma>0$.
\end{hp}

We denote by $(P_\varepsilon)_\varepsilon$ the family of operators defined by 
\begin{equation}\label{eq.Peps}
\begin{split}
P_\varepsilon&=D_t-A_\varepsilon-B_\varepsilon-V_\varepsilon\\
&:=D_t-\sum_{i,j=1}^nD_{x_i}(a_{ij,\varepsilon}(x)D_{x_j})  -\sum_{k=1}^nb_{k,\varepsilon}(x)D_{x_k}-V_\varepsilon(x), \quad \varepsilon \in (0,1].
\end{split}
\end{equation}
Note that $B_\eps=\mathrm{Op}(b_\eps)$, where for each $\eps \in (0,1]$, with the same notation as above, one has
\begin{equation}\label{eq.beps}
|\partial_x^\beta \partial_\xi^\alpha b_\eps(x,\xi)|\leq C_{\alpha \beta}\omega(\eps)^{-|\beta|-N_1}\langle \xi \rangle^{1-|\alpha|}, \quad \forall \alpha,\beta \in \mathbb{N}_0^n, \ (x,\xi) \in \R^n \times \R^n.
\end{equation}
Henceforth, we say that $(P_\eps)_\eps$ satisfies condition \ref{hp1}, \ref{hp2}, \ref{hp3}, \ref{hp4}, \ref{hp5}, which means that the nets of coefficients of $(P_\eps)_\eps$ we consider satisfy such conditions.

Before discussing in detail the notion of weak solution for problems of the form \eqref{mainprobIntro} (i.e. the Cauchy problem induced by the operator $P$ defined in \eqref{def.P}), we prove the following theorem, which provides \textit{sufficient} conditions for the singular operator $P$ to generate a net of operators satisfying conditions \ref{hp1}-\ref{hp5}.

\begin{Th}
\label{theo_with_sc}
Let $P$ be an operator of the form 
\[
  P=D_t-\sum_{i,j=1}^nD_{x_i}(a_{ij}(x)D_{x_j})- \sum_{k=1}^nb_k(x)D_{x_k}-V(x),
\]
and assume that the coefficients fulfill the following hypotheses.
\begin{enumerate}[label=(\roman*)]
    \item The matrix $\mathcal{A}(x)=(a_{ij}(x))_{i,j=1}^n$ is real and symmetric, for all $ x \in \R^n$ and its entries can be decomposed as
    \[
    a_{ij}(x)=c_{ij}+\tilde{a}_{ij}(x), \quad  i,j=1,\dots,n,
    \]
    where $c_{ij} \in \R$ and $\tilde{a}_{ij} \in W^{1,\infty}(\R^n)$, for all $i,j=1,\dots,n$. Moreover, we assume that the matrix  $\mathcal{C}=(c_{ij})_{i,j=1}^n$ is non-degenerate, namely there exists $\mu'>0$ such that  
    \[
    {\mu'}^{-1}|\xi|\leq |\mathcal{C} \cdot \xi |\leq \mu' |\xi|, \quad \forall \xi \in \R^n.
    \]
    Finally, the perturbations are assumed to have the form  
    \begin{equation}\label{eq.smalltildeaij}
    \tilde{a}_{ij}(x)=\langle x \rangle^{-N}\tilde{a}_{ij}'(x),\quad i,j=1,\dots,n,
    \end{equation}
    where $N \in \mathbb{N}$, $N>1$ and $\tilde{a}_{ij}' \in W^{1,\infty}(\R^n)$, such that $\|\tilde{a}_{ij}'\|_{W^{1,\infty}}\leq \nu'$, for some $\nu'>0$ sufficiently small, for all $i,j=1,\dots,n$.
    \item For all $k=1,\dots,n$, we assume $b_k \in \mathscr{E}'(\R^n)$
    and (with $N>1$ as above)
    \[
    \mathrm{Im}(b_k)=\langle x \rangle^{-N}b_k',
    \]
    where $\|b_{k}'\|_{L^\infty} \leq c_0'$, for some universal constant $c_0'>0$.
    \item $V \in \mathscr{E}'(\R^n)$. 
\end{enumerate}
Then, the corresponding net of regularised operators $(P_\eps)_\eps$ defined as in Section \ref{sec.stateandreg}, fulfills \ref{hp1}-\ref{hp5}.
\end{Th}
\begin{proof}
 The hypothesis \ref{hp1} easily follows from the fact that $\mathcal{A}(x)$ is real and symmetric (for all $x$) and from the definition of the mollifier (cf. \eqref{eq.defphiepsilon}).
 
 \noindent We now prove \ref{hp2}.
 Since 
 \[
 a_{ij}=c_{ij}+\tilde{a}_{ij}, \quad i,j=1,\dots,n,
 \]
 denoting by $\tilde{\mathcal{A}}_\eps$ the matrix with entries $\tilde{a}_{ij,\eps}=\tilde{a}_{ij}\ast \varphi_{\omega(\eps)}$, we have 
 \begin{equation}\label{eq.Aepssc}
|\mathcal{A}_\varepsilon(x)\cdot \xi|=|(\mathcal{C}+\tilde{\mathcal{A}}_\eps(x))\cdot \xi| \geq |\mathcal{C} \cdot \xi |-|\tilde{\mathcal{A}}_\eps(x) \cdot \xi|\geq {\mu'}^{-1} |\xi|-|\tilde{\mathcal{A}}_\eps(x) \cdot \xi|.
 \end{equation}
Next, note that
\[
|\tilde{\mathcal{A}}_\eps(x) \cdot \xi|^2=\sum_{i=1}^n\Bigl(\sum_{j=1}^n\tilde{a}_{ij,\eps}(x)\xi_j\Bigr)^2\leq C \sum_{i=1}^n\sum_{j=1}^n\tilde{a}_{ij,\eps}^2(x)\xi_j^2\leq C \max_{i,j}\|\tilde{a}_{ij,\eps}\|^2_{L^\infty}|\xi|^2,
\]
for some $C=C(n)>0$, and that, by \eqref{eq.smalltildeaij},
\[
\|\tilde{a}_{ij,\eps}\|_{L^\infty} \leq \|\varphi_{\omega(\eps)}\|_{L^1} \|\tilde{a}_{ij}\|_{L^\infty} \leq \nu',\quad \forall i,j=1,\dots,n. 
\]
Therefore, choosing $\nu'>0$ small enough, from \eqref{eq.Aepssc}, we obtain
\[
|\mathcal{A}_\eps(x)\cdot \xi|\geq {\mu'}^{-1} |\xi|-C\nu' |\xi|\geq \mu_1 |\xi|, \quad \forall \xi \in \R^n,
\]
for some $\mu_1>0$ independent of $\eps$. By repeating the same argument we also have 
\[
|\mathcal{A}_\eps(x)\cdot \xi |\leq |\mathcal{C} \cdot \xi|+|\tilde{\mathcal{A}}_\eps \cdot \xi|\leq \mu' |\xi|+C\nu' |\xi|\leq \mu_2|\xi|, \quad \forall \xi \in \R^n, 
\]
for some $\mu_2>0$ independent of $\eps$. Hence \ref{hp2} holds. 

\noindent The hypothesis \ref{hp3} follows directly from Remark \ref{rem.satshp3}. As for \ref{hp4} and \ref{hp5} we have that, for all $k=1,\dots, n$, 
\[
\begin{split}
|\mathrm{Im}(b_k) \ast \varphi_{\omega(\eps)}(x)|&\leq \int_{\R^n}|\langle x-y \rangle^{-N}b_k'(x-y)\varphi_{\omega(\eps)}(y)|dy\\
&\leq \langle x \rangle^{-N} \int |b_k'(x-y)\langle y \rangle^N \varphi_{\omega(\eps)}(y)|dy\\
&\leq C\langle x \rangle^{-N}\|b_{k}'\|_{L^\infty}\\
&\leq c_0 \langle x \rangle^{-N}, \quad \forall x \in \R^n,
\end{split}
\]
where the constants $C,c_0>0$ depend on the mollifier but are independent of $\eps$. Finally, for all $k=1,\dots,n$, since $b_k \in \mathscr{E}'(\R^n)$, by Proposition 2.1 in \cite{AACG}, we get
\[
|\bigl(\partial_x^\alpha (b_k \ast \varphi_{\omega(\eps)})\bigr)(x)|=|\bigl(b_k \ast \partial_x^\alpha \varphi_{\omega(\eps)}\bigr)(x)|\leq C_{\alpha}\omega(\eps)^{-|\alpha|-N_1}, \quad \forall x \in \R^n,
\]
and analogously 
 \[
|\bigl(\partial_x^\alpha (V \ast \varphi_{\omega(\eps)})\bigr)(x)|=|\bigl(V \ast \partial_x^\alpha \varphi_{\omega(\eps)}\bigr)(x)|\leq C_{\alpha}'\omega(\eps)^{-|\alpha|-N_2}, \quad \forall x \in \R^n, 
\]
for all $\alpha \in \mathbb{N}^n$ and for $C_\alpha,C_\alpha'>0$, $N_1, N_2 \in \mathbb{N}$, independent of $\eps$. Therefore \ref{hp4} and \ref{hp5} hold and the proof ends.
\end{proof}
\begin{Rem}
    Note that, in Theorem \ref{theo_with_sc}, it would be possible to increase the magnitude of the perturbations $\tilde{a}_{ij}$ and require an $L^\infty$-bound (depending on the size of the constant coefficients $c_{ij}$) given by a larger constant than the $\nu$ used for the derivatives (cf. Remark \ref{rmk.diag}). However, to keep the presentation as clear as possible, we have chosen to require the coefficients and their derivatives to be bounded by the same constant in the $L^\infty$-norm.
\end{Rem}

Our goal now is to find a $H^\infty$-very weak solution of problem \eqref{mainprobIntro}, which is based on the notion of moderate net of functions as stated in Subsection \ref{subsec.modnet} in the following sense. For the sake of completeness we state also the notion of $H^s$-very weak solution that will be useful for some class of problems (see Remark \ref{rmk.Hs} below).
\begin{Def}\label{def.veryweak}
The net $(u_\varepsilon)_\varepsilon \in \lbrace C([0,T],H^\infty(\mathbb{R}^n))\rbrace^{(0,\eps_0]}$ (resp. $(u_\varepsilon)_\varepsilon \in \lbrace C([0,T],H^s(\mathbb{R}^n))\rbrace^{(0,\eps_0]}$) \textit{is a $H^\infty$-very weak solution (resp. $H^s$-very weak solution)} of \eqref{mainprobIntro},
if there exist $(u_{0,\varepsilon})_{\varepsilon}$ and $(f_\varepsilon)_\varepsilon$ $H^\infty$-moderate (resp. $H^s$-moderate) regularisations of $u_0$ and $f$, respectively, such that, for any $\varepsilon \in (0,\eps_0]$,  $u_\varepsilon$ solves 
\begin{equation}\label{probveryweak}
    \begin{cases}
        P_\varepsilon u_\varepsilon=f_\varepsilon \quad \text{in} \ \ (0,T]\times \mathbb{R}^n, \\
        u_\varepsilon(0,\cdot)=u_{0,\varepsilon} \quad \text{in} \ \ \mathbb{R}^n,
    \end{cases}
\end{equation}
 and $(u_\varepsilon)_\varepsilon$ is $H^\infty$-moderate (resp. $H^s$-moderate).
\end{Def}

\subsection{Regularised Problem}
In order to study the weak well-posedness of problem \eqref{mainprobIntro}, the first step is to consider, for fixed $\eps \in (0,1]$,  the \textit{$\eps$-regularised problem} (cf. \cite{AACG})
\begin{equation}\label{eq.reg}
\begin{cases}
        P_\varepsilon u=g \quad \text{in} \ \ (0,T]\times \mathbb{R}^n, \\
        u_\varepsilon(0,\cdot)=v_0 \quad \text{in} \ \ \mathbb{R}^n,
    \end{cases}
\end{equation}
where $v_0 \in H^s(\mathbb{R}^n)$ and $g$ has suitable regularity. 

The key result to obtain the existence and uniqueness of a $H^\infty$-very weak solution is Theorem~\ref{theo.apriori} below. Indeed, the estimates we derive in this theorem for the solution of \eqref{eq.reg} (for each fixed $\varepsilon$) ensure the moderateness of the net $(u_\varepsilon)_\varepsilon$, as well as the uniqueness (in the sense explained in Subsection \ref{sub.uniq} below) of a $H^\infty$-very weak solution, which we discuss in detail in Section~\ref{sec.exandun}.

The proof of Theorem \ref{theo.apriori} follows the same pattern as Lemma 4.1 in \cite{FT} (cf. \cite{KPRV}, \cite{FS}). We refer the reader to \cite{FT} for further details.
\begin{Th}\label{theo.apriori}
    Let $\varepsilon \in (0,1]$ and consider the regularised problem \eqref{eq.reg} with Cauchy data $v_0 \in H^s(\R^n)$ and assume that \ref{hp1}-\ref{hp5} hold for $(P_\eps)_\eps$. 
    
    Then, we have:
    \begin{itemize}
    \item [(i)] If $g\in L^1([0,T];H^s(\R^n))$, the IVP \eqref{eq.reg} has a unique solution $u_\varepsilon\in C([0,T];H^s(\R^n))$ satisfying    
    \begin{equation}\label{eq.i}
    \sup_{0\leq t\leq T}\|u_\varepsilon(t,\cdot)\|_s\leq C_2e^{C_1\omega(\eps)^{-k_1}T}\left(\|v_0\|_s+\int_0^T \|g(t,\cdot)\|_sdt\right).
\end{equation}
\item [(ii)]  If $g\in L^2([0,T];H^s(\R^n))$, the IVP \eqref{eq.reg} has a unique solution $u_\varepsilon\in C([0,T];H^{s}(\R^n))$ satisfying
\begin{equation}
\begin{split}
    &\sup_{0\leq t \leq T}\|u_\varepsilon(t,\cdot)\|^2_s+\omega(\eps)^{k}\int_0^T \|\langle x \rangle^{-{N/2}}\Lambda^{s+1/2}u_\eps(t,\cdot)\|_0^2\,dt\\
    &\hspace{4cm} \leq C_2e^{C_1\omega(\eps)^{-k_1}T}\left(\|v_0\|_s^2+\int_0^T \|g(t,\cdot)\|_s^2 dt\right).
\end{split}
\end{equation}
\item [(iii)]  If $\Lambda^{s-\frac{1}{2}}g\in L^2([0,T]\times \R^n;\langle x \rangle^N dx dt)$, the IVP \eqref{eq.reg} has a unique solution $u_\varepsilon\in C([0,T];H^{s}(\R^n))$ satisfying
\begin{equation}
\begin{split}
    &\sup_{0\leq t \leq T}\|u_\varepsilon(t,\cdot)\|^2_s+\omega(\eps)^{k} \int_0^T \|\langle x \rangle^{-{N/2}}\Lambda^{s+1/2}u_\eps(t,\cdot)\|_0^2\,dt\\
    & \hspace{4cm} \leq C_2e^{C_1\omega(\eps)^{-k_1}T}\left(\|v_0\|_s^2+\int_0^T \| \langle x \rangle^{N/2}\Lambda^{s-1/2}g(t,\cdot)\|^2_0 \,dt\right),
    \end{split}
\end{equation}
\end{itemize}  
where $C_1,C_2>0$ and $k, k_1 \in \mathbb{N}_0$, depend on the net of coefficients of $(P_\eps)_\eps$ defined in \eqref{eq.Peps} and on $s$, but are independent of $\eps$.
\end{Th}
\begin{proof}
The idea of the proof, following the approach in \cite{FT}, is to define, for each fixed $\varepsilon$, a norm $N^s_\varepsilon(u)$ on the Sobolev space $H^s(\R^n)$ that is equivalent to the standard norm $\|\cdot\|_s$ (see \cite{FT}), and then to estimate the time derivative of this norm, keeping track of the explicit dependence on $\varepsilon$.

To this end, fix $\varepsilon \in (0,1]$ and $s \in \R$. Define, for $\varepsilon \in (0,1]$ the operator $E_\varepsilon=\mathrm{Op}(e^{d_\eps})\in \Psi^0$, where $d_\eps \in S^0$ is the symbol constructed in Lemma \ref{lemma.Doiepsilon} satisfying \eqref{eq.pDoi}.
We then introduce the quantity
    \begin{equation*}
    N^s_\varepsilon(u):=\bigl(\|E_\varepsilon \Lambda^su\|_0^2+\|u\|_{s-1}^2)^{1/2}.
    \end{equation*}
    It is also convenient to denote $\tilde{E}_\varepsilon=\mathrm{Op}(e^{-d_\eps})$ and write $\tilde{E}_\varepsilon E_\eps=I+R_{-1,\varepsilon}$, with $R_{-1,\varepsilon}=\mathrm{Op}(r_{-1,\varepsilon})$, $r_{-1,\varepsilon} \in S^{-1}$. 
    By Theorem \ref{theo.bound} and Theorem \ref{theo.comp} we get
    \begin{equation}\label{eq.normeq1}
    \begin{split}
        \|u\|_s^2 &\leq \|\tilde{E}_\varepsilon E_\varepsilon\Lambda^su\|_0^2+\|R_{-1,\varepsilon}\Lambda^su\|_0^2\\
        &\leq C_1'(|e^{-d_\eps}|_{k_1'}^{(0)}\|E_\varepsilon \Lambda^s u\|_0^2+|r_{-1,\varepsilon}|_{k_1''}^{(-1)}\|u\|_{s-1}^2) \\
        &\leq C_1''\Bigl(|e^{-d_\eps}|_{k_1'}^{(0)}+|e^{d_\eps}|_{k_1'''}^{(0)}|e^{-d_\eps}|^{(0)}_{k_1'''}\Bigr)N^s_\varepsilon(u)^2\\
        &\leq C_1 \omega(\varepsilon)^{-k_1}N_\varepsilon^s(u)^2, \quad u \in H^s(\R^n),
    \end{split}
    \end{equation}
    where $C_1,C_1',C_1''>0$ $k_1, k_1',k_1'',k_1''' \in \mathbb{N}_0$, depend on $s$ and on the dimension $n$, but are independent of $\varepsilon$. Similarly,
    \begin{equation}\label{eq.normeq2}
    \begin{split}
        N_\varepsilon(u)^2 &\leq (C_2'|e^{d_\eps}|_{k_2'}^{(0)}+1)\|u\|_s^2 \\
        &\leq C_2\omega(\varepsilon)^{-k_2}\|u\|_s^2, \quad u \in H^s(\R^n),
        \end{split}
    \end{equation}
    with $C_2,C_2'>0$, $k_2,k_2' \in \mathbb{N}_0$, independent of $\eps$ (but, as before depend on $s$ and on $n$). Therefore, \eqref{eq.normeq1} and \eqref{eq.normeq2} show that the two norms $N_\varepsilon^s(\cdot)$, and $\|\cdot \|_s$ are indeed equivalent, and provide precise information on the dependence of equivalent constants on $\varepsilon$.  
    
    Let now $u_\varepsilon \in C([0,T];H^s(\R^n))$ be the solution of \eqref{eq.reg} (in the case of $g \in L^1([0,T],H^s(\R^n)$). For the existence of such a $u_\eps$ we refer to \cite{FT} (cf. \cite{KPRV} and Remark \ref{rmk.class} below). Our aim is to estimate $\partial_t N_\varepsilon^s(u_\varepsilon)^2$. By integrating in time, and using the equivalence shown in \eqref{eq.normeq1} and \eqref{eq.normeq2}, this estimate will lead to the desired estimate $(i)$, $(ii)$ and $(iii)$.
    
    \noindent To do so, we start by examining the term $\partial_t \|u_\varepsilon\|_{s-1}^2$. By using the self-adjoint property of $A_\varepsilon$, and once again Theorem \ref{theo.bound} and Theorem \ref{theo.comp}, we get
    \begin{equation*}
    \begin{split}
        \partial_t \|u_\varepsilon\|_{s-1}^2&=2\mathrm{Re}(\Lambda^{s-1}\partial_tu_\varepsilon,\Lambda^{s-1}u_\varepsilon)_0 \\
        &=2\mathrm{Re}(\Lambda^{s-1}(i(A_\eps+B_\eps+V_\eps)u_\eps+ig),\Lambda^{s-1}u_\eps)_0 \\
        &=2\mathrm{Re}(i\Lambda^{s-1}A_\varepsilon u_\varepsilon,\Lambda^{s-1}u_\varepsilon)_0+2\mathrm{Re}(i\Lambda^{s-1}(B_\eps+V_\eps)u_\eps,\Lambda^{s-1}u_\eps)_0+2\mathrm{Re}(i\Lambda^{s-1}g,\Lambda^{s-1}u_\varepsilon)_0  \\
        &=2 \mathrm{Re}(i[\Lambda^{s-1},A_\varepsilon] u_\varepsilon,\Lambda^{s-1}u_\varepsilon)_0+2\mathrm{Re}(i\Lambda^{s-1}(B_\eps+V_\eps)u_\eps,\Lambda^{s-1}u_\eps)_0+2\mathrm{Re}(i\Lambda^{s-1}g,\Lambda^{s-1}u_\varepsilon)_0 \\
        & \leq C'\Bigl(\bigl(|a_\eps|_{k'}^{(2)}+|b_\eps|^{(1)}_{k''}+|V_\eps|^{(0)}_{k'''}\bigr)\|u_\eps\|_s^2+\|g\|_s\|u_\eps\|_s\Bigr).
    \end{split}
    \end{equation*}
    Therefore, by \eqref{eq.conda1}, \eqref{eq.conda2}, \eqref{eq.Veps}, \eqref{eq.beps}, we obtain
    \begin{equation}\label{eq.smeno1}
        \partial_t \|u_\varepsilon\|_{s-1}^2 \leq C \Bigl(\omega(\eps)^{-k}\|u_\eps\|_s^2+\|g\|_s\|u_\eps\|_s\Bigr),
    \end{equation}
    for some $k=k(n,s)$ and some $C=C(n,s)>0$ independent of $\eps$.
    We next estimate the term $\partial_t\|E_\eps \Lambda^su_\eps\|_0^2$. By using again the self-adjoint property of $A_\eps$ we have 
    \begin{equation}\label{eq.1piu2piu3}
    \begin{split}
        \partial_t\|E_\eps \Lambda^su_\eps\|_0^2&=2\mathrm{Re}(E_\eps \Lambda^s\partial_tu_\eps,E_\eps\Lambda^su_\eps)_0 \\
        &\leq 2\mathrm{Re}(E_\eps \Lambda^s iA_\eps u_\eps,E_\eps\Lambda^s u_\eps)_0+2\mathrm{Re}(E_\eps \Lambda^s i(B_\eps+V_\eps) u_\eps,E_\eps\Lambda^su_\eps)_0\\
        &\quad + 2\mathrm{Re}(E_\eps \Lambda^s ig,E_\eps\Lambda^su_\eps)_0\\
        &\leq 2\mathrm{Re}([E_\eps \Lambda^s, iA_\eps] u_\eps,E_\eps\Lambda^s u_\eps)_0+2\mathrm{Re}(E_\eps \Lambda^s i(B_\eps+V_\eps) u_\eps,E_\eps\Lambda^su_\eps)_0\\
        &\quad + 2\mathrm{Re}(E_\eps \Lambda^s ig,E_\eps\Lambda^su_\eps)_0\\
        &=(I)+(II)+(III).
    \end{split}
    \end{equation}
    Hence, to estimate $\partial_t\|E_\eps \Lambda^su_\eps\|_0^2$ we estimate $(I)$, $(II)$ and $(III)$ separately. 

    For $(I)$, as in \cite{FT}, we have
    \begin{equation*}
        (I)=2\mathrm{Re}([E_\eps,iA_\eps]\Lambda^su_\eps,E_\eps \Lambda^s u_\eps)_0+2\mathrm{Re}(E_\eps[\Lambda^s,iA_\eps]u_\eps,E_\eps\Lambda^su_\eps)_0.
    \end{equation*} 
    At this point, we note that, by Theorem \ref{theo.comp}
    \begin{equation*}
    \begin{split}
[E_\varepsilon,iA_\eps]&=\mathrm{Op}(\lbrace e^{d_\eps},a_\eps\rbrace)+\mathrm{Op}(r^1_{0,\eps})\\
        &=\mathrm{Op}(\lbrace d_\eps,a_\eps \rbrace e^{d_\eps})+\mathrm{Op}(r^1_{0,\eps})\\
        &=\mathrm{Op}(-\lbrace a_\eps, d_\eps\rbrace)E_\eps+\mathrm{Op}(r^2_{0,\eps}),
        \end{split}
    \end{equation*}
    and for all $k \in \mathbb{N}_0$, there exist $k_1 \in \mathbb{N}_0$ and $C=C(k)>0$, such that     
    \begin{equation*}
|r^1_{0,\eps}|_k^{(0)}+|r^2_{0,\eps}|_{k}^{(0)} \leq C |a_\eps|_{k_1}^{(2)}\Bigl( |d_\eps|_{k_1}^{(0)}+(|d_\eps|_{k_1}^{(0)})^2\Bigr).
    \end{equation*}
    Moreover, 
    \begin{equation}
    \begin{split}
E_\eps[\Lambda^s,iA_\eps]&=E_\eps[\Lambda^s,iA_\eps]\Lambda^{-s}\tilde{E}_\eps E_\eps \Lambda^s+\mathrm{Op}(r^3_{0,\eps})\\
&=\mathrm{Op}(\lbrace \langle \xi \rangle^{s},a_\eps \rbrace \langle \xi \rangle^{-s})E_\eps\Lambda^s+\mathrm{Op}(r_{s,\eps}^1) 
\end{split}
    \end{equation}
    where, once again, for any $k \in \mathbb{N}_0$, there exist $k_1' \in \mathbb{N}_0$ and $C'=C'(k)>0$, such that     
    \begin{equation*}
|r^3_{0,\eps}|_k^{(0)}+|r^1_{s,\eps}|_k^{(s)} \leq C' |a_\eps|_{k_1'}^{(2)}|d_\eps|_{k_1'}^{(0)}.
    \end{equation*}
    Thus (recall $H_{a_\eps}d_\eps=\lbrace a_\eps, d_\eps \rbrace$)
    \begin{equation*}
    \begin{split}
        (I)&= -2\mathrm{Re}(\mathrm{Op}(H_{a_\eps}d_\eps -\lbrace \langle \xi \rangle^{s},a_\eps \rbrace \langle \xi \rangle^{-s})E_\eps\Lambda^su_\eps,E_\eps\Lambda^su_\eps)_0+2\mathrm{Re}(\mathrm{Op}(r_{s,\eps}^2)u_\eps,E_\eps\Lambda^su_\eps)_0
        \end{split}
    \end{equation*}
    with $r_{s,\eps}^2 \in S^s$ having seminorms bounded by the seminorms of $a_\eps$ and $d_\eps$, as before. 

For $(II)$ we have 
\begin{equation}
    (II)=2\mathrm{Re}(i(B_\eps+V_\eps)E_\eps\Lambda^su_\eps,E_\eps \Lambda^su_\eps)_0+(\mathrm{Op}(r^3_{s,\eps})u_\eps,E_\eps\Lambda^s u_\eps)_0,
\end{equation}
and then 
\begin{equation}\label{eq.1piu2}
\begin{split}
    (I)+(II)&=2\mathrm{Re}(\mathrm{Op}(-H_{a_\eps}d_\eps +\lbrace \langle \xi \rangle^{s},a_\eps \rbrace \langle \xi \rangle^{-s}+ib_\eps)E_\eps\Lambda^su_\eps,E_\eps\Lambda^su_\eps)_0 \\
    &\quad +2\mathrm{Re}(iV_\eps E_\eps\Lambda^su_\eps,E_\eps \Lambda^su_\eps)_0+(\mathrm{Op}(r^2_{s,\eps}+r^3_{s,\eps})u_\eps,E_\eps\Lambda^s u_\eps)_0.
\end{split}
\end{equation}
The idea now is to use the Sharp G\aa rding inequality (see Theorem \ref{theo.ShGa}) to estimate the first term in \eqref{eq.1piu2} . To this end, we use the hypotheses \ref{hp2}, \ref{hp4} and \eqref{eq.pDoi} to obtain (recall $B_\eps=\mathrm{Op}(b_\eps)$)
\begin{equation*}
\begin{split}
    \mathrm{Re}\Bigl(-H_{a_\eps}d_\eps +\lbrace \langle \xi \rangle^{s},a_\eps \rbrace \langle \xi \rangle^{-s}+ib_\eps\Bigr)&=-H_{a_\eps}d_\eps +\lbrace \langle \xi \rangle^{s},a_\eps \rbrace \langle \xi \rangle^{-s}-\mathrm{Im}(b_\eps)\\
    &\leq -C_1\langle x \rangle^{-N}|\xi|+C_2+C_3\nu \langle x\rangle^{-N}|\xi|+c_0\langle x \rangle^{-N}|\xi|,
    \end{split}
\end{equation*}
with $C_2,C_3,\nu,c_0>0$ universal constants and $C_1>0$ can be taken arbitrarily large (see Remark \ref{rmk.C1}). Hence, by using $(1+|\xi|^2)^{1/2}\leq 1+|\xi|$, 
\begin{equation*}
    \mathrm{Re}\Bigl(-H_{a_\eps}d_\eps +\lbrace \langle \xi \rangle^{s},a_\eps \rbrace \langle \xi \rangle^{-s}+ib_\eps\Bigr) \leq -C\langle x \rangle^{-N}\langle\xi\rangle+C_2',
\end{equation*}
with $C,C_2'>0$ universal constants independent of $\eps$. Thus, by Theorem \ref{theo.ShGa}, we get  
\begin{equation*}
\begin{split}
    (I)+(II) &\leq -2C\mathrm{Re}(\langle x \rangle^{-N}\Lambda E_\eps \Lambda^{s} u_\eps,E_\eps\Lambda^s u_\eps)_0+C_2'\|E_\eps\Lambda^su_\eps\|_0^2 \\
    &\quad +2\mathrm{Re}(iV_\eps E_\eps\Lambda^su_\eps,E_\eps \Lambda^su_\eps)_0+(\mathrm{Op}(r^2_{s,\eps}+r^3_{s,\eps})u_\eps,E_\eps\Lambda^s u_\eps)_0 \\
    &\leq -2C\mathrm{Re}(\langle x \rangle^{-N}\Lambda E_\eps \Lambda^{s} u_\eps,E_\eps\Lambda^s u_\eps)_0+C_1 \omega(\eps)^{-k_1}N_\eps^s(u_\eps)^2,
\end{split}
\end{equation*}
for some $C,C_1>0$ and $k_1 \in \mathbb{N}_0$, possibly differing from previous occurrences but independent of $\eps$.
We now note that (cf. \cite{KPRV}, p. 393)
\begin{equation*}
    \langle x \rangle^{-N} \Lambda=\mathrm{Op}(\langle x \rangle^{-N}\langle \xi \rangle)=\mathrm{Op}(\langle x \rangle^{-N/2}\langle \xi \rangle^{1/2})\mathrm{Op}(\langle x \rangle^{-N/2}\langle \xi \rangle^{1/2})+\mathrm{Op}(r_0), \quad r_0 \in S^0,
\end{equation*}
and, by Theorem \ref{theo.adjoint},
\begin{equation*}
    \mathrm{Op}(\langle x \rangle^{-N/2}\langle \xi \rangle^{1/2})=\bigl(\mathrm{Op}(\langle x \rangle^{-N/2}\langle \xi \rangle^{1/2})\bigr)^\ast + \mathrm{Op}(r_{-1/2}), \quad r_{-1/2} \in S^{-1/2}.
\end{equation*}
Thus, possibly enlarging the universal constant $C_1$, we get 
\begin{equation}\label{eq.quasifinalIandII}
    (I)+(II) \leq -2C\mathrm{Re}(\langle x \rangle^{-N/2}\Lambda^{1/2} E_\eps \Lambda^{s} u_\eps,\langle x \rangle^{-N/2}\Lambda^{1/2}E_\eps\Lambda^s u_\eps)_0+C_1 \omega(\eps)^{-k_1}N_\eps^s(u_\eps)^2.
\end{equation}
Now, note that by using once again the properties of the calculus stated in Theorem \ref{theo.bound} and Theorem \ref{theo.comp}, we get
\begin{equation}\label{eq.us}
\begin{split}
\|\langle x \rangle^{-{N/2}}\Lambda^{s+1/2}u_\eps\|_0&=\|\langle x \rangle^{-N/2}\Lambda^{1/2}\tilde{E}_\eps E_\eps \Lambda^s u_\eps\|_0+\|\langle x\rangle^{-N/2}\Lambda^{1/2}R_{-1,\eps}\Lambda^{s}u_\eps\|_0\\
&\leq C'\omega(\eps)^{-k'}\|\langle x \rangle^{-N/2}\Lambda^{1/2}E_\eps \Lambda^s u_\eps\|_0+C''\omega(\eps)^{-k''}N_\eps^s(u_\eps),
\end{split}
\end{equation}
for some universal constants $C',C''>0$ and $k',k'' \in \mathbb{N}_0$.

Therefore, by combining \eqref{eq.quasifinalIandII} and \eqref{eq.us} with possibly different constants $C,C_1>0$ and $k,k_1 \in \mathbb{N}_0$, we obtain
\begin{equation}\label{eq.fin1piu2}
    (I)+(II) \leq  -C\omega(\eps)^{k}\|\langle x \rangle^{-{N/2}}\Lambda^{s+1/2}u_\eps\|_0^2+C_1\omega(\eps)^{-k_1}N_\eps^s(u_\eps)^2.
\end{equation}
The estimate of the term $(III)$ in \eqref{eq.1piu2piu3} depends on which of the estimate $(i)$, $(ii)$, or $(iii)$ in the statement we aim to obtain. In what follows, we denote by $C,C_1,C_2>0$ and $k,k_1,k_2 \in \mathbb{N}_0$ some constants, possibly changing from line to line, but always independent of $\eps$.  

\textit{Case $(i)$}.
In this case it is sufficient to note that 
\begin{equation}\label{eq.est3}
    (III)=2\mathrm{Re}(E_\eps \Lambda^s ig,E_\eps\Lambda^su_\eps)_0 \leq C_2\omega(\eps)^{-k_2}N_\eps^s(g)N_\eps^s(u_\eps).
\end{equation}
Hence, by \eqref{eq.smeno1}, \eqref{eq.1piu2piu3} and \eqref{eq.fin1piu2}, we get 
\begin{equation*}
    \partial_t N^s_\eps(u_\eps)^2 \leq C_1\omega(\eps)^{-k_1}N_\eps^s(u_\eps)^2+C_2\omega(\eps)^{-k_2}N_\eps^s(g)N_\eps^s(u_\eps)
\end{equation*}
and so 
\begin{equation*}
    \partial_t N^s_\eps(u_\eps) \leq C_1\omega(\eps)^{-k_1}N_\eps^s(u_\eps)+C_2\omega(\eps)^{-k_2}N_\eps^s(g),
\end{equation*}
that we may rewrite as 
\begin{equation*}
    \partial_t\Bigl(e^{-C_1\omega(\eps)^{-k_1}t}N_\eps^s(u_\eps)\Bigr)\leq e^{-C_1\omega(\eps)^{-k_1}t}C_2\omega(\eps)^{-k_2}N_\eps^s(g).
\end{equation*}
Therefore, by integrating in time and using the equivalence of the Sobolev norms $N_\eps^s(\cdot)$ and $ \|\cdot \|_s$ proved in \eqref{eq.normeq1} and \eqref{eq.normeq2}), we obtain that, for each $t \in [0,T]$,
\begin{equation}
\begin{split}
    \|u_\eps(t,\cdot)\|_s&\leq C_2\omega(\eps)^{-k_2}e^{C_1\omega(\eps)^{-k_1}T}\Bigl(\|v_0\|_s+\int_0^T \|g(t,\cdot)\|_s dt\Bigr)\\
    &\leq C_2e^{C_1\omega(\eps)^{-k_1}T}\Bigl(\|v_0\|_s+\int_0^T \|g(t,\cdot)\|_s dt\Bigr).
\end{split}
\end{equation}

\textit{Case $(ii)$.} To get $(ii)$ we use again \eqref{eq.smeno1}, \eqref{eq.fin1piu2} and \eqref{eq.est3}, to obtain
\begin{equation}\label{eq.usest2}
\begin{split}
    &\partial_t \Bigl(e^{-C_1\omega(\eps)^{-k_1}t}  N_\eps^s(u_\eps)^2 \Bigr)\\
    &\quad \leq e^{-C_1\omega(\eps)^{-k_1}t }\Bigl( -C\omega(\eps)^{k}\|\langle x \rangle^{-{N/2}}\Lambda^{s+1/2}u_\eps\|_0^2+C_2\omega(\eps)^{-k_2}N_\eps^s(g)N_\eps^s(u_\eps)\Bigr).
\end{split}
\end{equation}
Moreover, note that 
\begin{equation}\label{eq.useful}
    N_\eps^s(g)N_\eps^s(u_\eps)\leq \omega(\eps)N_\eps^s(u_\eps)^2+\omega(\eps)^{-1}N_\eps^s(g)^2.
\end{equation}
Hence, by reasoning as before, for each $t \in [0,T]$ we get
\begin{equation*}
\begin{split}
&\|u_\varepsilon(t,\cdot)\|^2_s+\omega(\eps)^{k}\int_0^T \|\langle x \rangle^{-{N/2}}\Lambda^{s+1/2}u_\eps(t,\cdot)\|_0^2\,dt\\
    &\hspace{4cm} \leq C_2\omega(\eps)^{-k_2}e^{C_1\omega(\eps)^{-k_1}T}\left(\|v_0\|_s^2+\int_0^T \|g(t,\cdot)\|_s^2 dt\right)\\
    & \|u_\varepsilon(t,\cdot)\|^2_s+\omega(\eps)^{k}\int_0^T \|\langle x \rangle^{-{N/2}}\Lambda^{s+1/2}u_\eps(t,\cdot)\|_0^2\,dt\\
    &\hspace{4cm} \leq C_2e^{C_1\omega(\eps)^{-k_1}T}\left(\|v_0\|_s^2+\int_0^T \|g(t,\cdot)\|_s^2 dt\right)\\
\end{split}
\end{equation*}
\textit{Case $(iii)$.} To obtain $(iii)$ we have to refine the estimate of the term $(III)$ in the following sense.
By using again \eqref{eq.useful}, we have (cf. \cite{FT})
\begin{equation*}
\begin{split}
    2\mathrm{Re}(E_\eps \Lambda^s ig,E_\eps\Lambda^su_\eps)_0&= 2\mathrm{Re}(\Lambda^{1/2}E_\eps \Lambda^{s-1/2} ig,E_\eps\Lambda^su_\eps)_0+2\mathrm{Re}([E_\eps,\Lambda^{1/2}] \Lambda^{s-1/2} ig,E_\eps\Lambda^su_\eps)_0\\
    &\leq 2\mathrm{Re}(E_\eps \Lambda^{s-1/2} ig,E_\eps\Lambda^{s+1/2}u_\eps)_0 + 2\mathrm{Re}(E_\eps \Lambda^{s-1/2} ig,[\Lambda^{1/2},E]\Lambda^su_\eps)_0\\
    &\quad + 2\mathrm{Re}([E_\eps,\Lambda^{1/2}] \Lambda^{s-1/2} g,E_\eps\Lambda^su_\eps)_0 \\
    & \leq 2\mathrm{Re}(\langle x \rangle^{N/2}E_\eps \Lambda^{s-1/2} ig,\langle x \rangle^{-N/2}E_\eps\Lambda^{s+1/2}u_\eps)_0 \\
    &\quad + 2\mathrm{Re}( \langle x \rangle^{N/2}E_\eps \Lambda^{s-1/2} ig,\langle x\rangle^{-N/2}[\Lambda^{1/2},E]\Lambda^su_\eps)_0\\
    &\quad + 2\mathrm{Re}(\langle x \rangle^{N/2}[E_\eps,\Lambda^{1/2}] \Lambda^{s-1/2} g,\langle x \rangle^{-N/2}E_\eps\Lambda^su_\eps)_0 \\
    &\leq C\omega(\eps)^{-k_2}\|\langle x \rangle^{N/2}\Lambda^{s-1/2}g\|_0^2+\omega(\eps)^{k}\Bigl(\|\langle x \rangle^{-N/2}\Lambda^{s+1/2}u_\eps\|_0^2+ N_\eps^s(u_\eps)^2 \Bigr).
\end{split}
\end{equation*}
Therefore, in this case
\begin{equation*}
\partial_t N_\eps^s(u_\eps)^2 \leq  -C\omega(\eps)^{k}\|\langle x \rangle^{-{N/2}}\Lambda^{s+1/2}u_\eps\|_0^2+C_1\omega(\eps)^{-k_1}N_\eps^s(u_\eps)^2+C_2\omega(\eps)^{-k_2}\|\langle x \rangle^{N/2}\Lambda^{s-1/2}g\|_0^2,
\end{equation*}
which leads (by repeating the same argument of Cases $(i)$ and $(ii)$) to the fact that, for all $t \in [0,T]$, we have
\begin{equation*}
\begin{split}
&\|u_\varepsilon(t,\cdot)\|^2_s+\omega(\eps)^{k} \int_0^T \|\langle x \rangle^{-{N/2}}\Lambda^{s+1/2}u_\eps(t,\cdot)\|_0^2\,dt\\
    & \hspace{4cm} \leq C_2\omega(\eps)^{-k_2}e^{C_1\omega(\eps)^{-k_1}T}\left(\|v_0\|_s^2+\int_0^T \| \langle x \rangle^{N/2}\Lambda^{s-1/2}g(t,\cdot)\|^2_0 \,dt\right)\\
    & \hspace{4cm} \leq C_2e^{C_1\omega(\eps)^{-k_1}T}\left(\|v_0\|_s^2+\int_0^T \| \langle x \rangle^{N/2}\Lambda^{s-1/2}g(t,\cdot)\|^2_0 \,dt\right)
    \end{split}
\end{equation*}
\end{proof}

We are now ready to employ this theorem to prove the \textit{existence} of a $H^\infty$-very weak solution of the Cauchy problem \eqref{mainprobIntro}.

\section{Existence and uniqueness of a $H^\infty$-very weak solution}\label{sec.exandun}
The aim of this section is first to prove the existence of a $H^\infty$-very weak solution in the sense of Definition \ref{def.veryweak} and then to establish a uniqueness result for such a solution. As anticipated above, both results follow from the estimates obtained in Theorem \ref{theo.apriori}.
Note that when we choose singular initial data and right-hand side their regularisation naturally leads to a $H^\infty$-very weak solution. If we choose instead initial data and right-hand side in $H^s(\R^n)$ and we do not regularise them, then we will generate a $H^s$-very weak solution (see Remark \ref{rmk.Hs} below).

\subsection{Existence}
We start by studying the existence (recall Definition \ref{def.veryweak}). The main result of this work is the following theorem, which, once again, relies on the estimates obtained in Theorem \ref{theo.apriori}.
\begin{Th}\label{theo.existence}
Let the regularisation $(P_\varepsilon)_\varepsilon$ of the operator $P$ in \eqref{mainprobIntro} fulfill the hypotheses \ref{hp1}-\ref{hp5} and let $ f \in C([0,T];H^{-\infty}(\R^n))$ and $u_0 \in H^{-\infty}(\R^{n})$. Then there exists a $H^\infty$-very weak solution $(u_\eps)_\eps$ of the problem \eqref{mainprobIntro}.
\end{Th}

\begin{proof}
Let $\varphi \in \mathscr{S}(\R^n)$, with $\int \varphi \, dx=1$. For $\eps \in (0,1]$ denote by $f_\eps(t,x)=(f(t,\cdot) \ast \varphi_\eps)(x)$ and $u_{0,\eps}(x)=(u_0 \ast\varphi_\eps)(x)$.
If $f \in C([0,T];H^{-\infty}(\R^n))$ and $u_0 \in H^{-\infty}(\R^n)$ we know (see Remark \ref{rmk.h-infty}) that for each $s \in \R$ there exist $N_f,N_{u_0} \in \mathbb{N}_0$ and $C>0$ such that 
\begin{equation*}
    \|f_\eps\|_s\leq C \eps^{-N_f}, \quad \|u_{0,\eps}\|_s \leq C \eps^{-N_{u_0}},
\end{equation*}
 uniformly in $t \in [0,T]$ and $\eps \in (0,1]$. Hence, by \eqref{eq.i},
\begin{equation*}
\begin{split}
    \sup_{0\leq t\leq T}\|u_\varepsilon(t,\cdot)\|_s&\leq C_2e^{TC_1\omega(\eps)^{-k_1}}\left(\|u_{0,\eps}\|_s+\int_0^T \|f_\eps(t,\cdot)\|_sdt\right)\\
        &\leq C_2'e^{TC_1\omega(\eps)^{-k_1}}\eps^{-N_f-N_{u_0}}.
\end{split}
\end{equation*}
Therefore, using the positive scale $(\omega_\eps)_\eps$ defined by 
\begin{equation*}
    \omega(\eps):=\log(\log(\eps^{-1}))^{-1}, \quad \eps \in (0,\eps_0],
\end{equation*}
for $\eps_0>0$ sufficiently small, 
we have the $H^\infty$-moderateness of $(u_\eps)_\eps$, and then the existence of a $H^\infty$-very weak solution, follows. Note that the scale $\omega$ has been chosen to be independent of $k_1$ and therefore independent of the Sobolev order $s$, leading to $H^\infty$-moderateness. 
\end{proof}

\begin{Rem}
Note that, under the assumptions of Theorem \ref{theo.existence}, we have that the net of solutions $(u_\eps)_\eps$ generated, for each fixed $s \in \R$, satisfies the following estimates with respect to the nets $(f_\eps)_\eps, (u_{0,\eps})_{\eps}$ associated with the Cauchy data $f,u_0$ (see the proof of Theorem \ref{theo.existence}). 
\begin{itemize}
    \item [(i)] If $(f_\eps)_\eps\in \lbrace L^1([0,T];H^s(\R^n))\rbrace^{(0,1]}$ 
    \begin{equation*}
    \sup_{0\leq t\leq T}\|u_\varepsilon(t,\cdot)\|_s\leq Ce^{C_1\omega(\eps)^{-k_1}T}\left(\|u_{0,\eps}\|_s+\int_0^T \|f_\eps(t,\cdot)\|_sdt\right).
\end{equation*}
\item [(ii)]  If $(f_\eps)_\eps\in \lbrace L^2([0,T];H^s(\R^n))\rbrace^{(0,1]}$
\begin{equation*}
\begin{split}
    &\sup_{0\leq t \leq T}\|u_\varepsilon(t,\cdot)\|^2_s+\omega(\eps)^{k}\int_0^T \|\langle x \rangle^{-{N/2}}\Lambda^{s+1/2}u_\eps(t,\cdot)\|_0^2\,dt\\
    &\hspace{4cm} \leq Ce^{C_1\omega(\eps)^{-k_1}T}\left(\|u_{0,\eps}\|_s^2+\int_0^T \|f_\eps(t,\cdot)\|_s^2 dt\right).
\end{split}
\end{equation*}
\item [(iii)]  If $(\Lambda^{s-\frac{1}{2}}f_\eps)_\eps\in \lbrace L^2([0,T] \times \R^n, \langle x \rangle^{-N}dxdt)\rbrace^{(0,1]}$
\begin{equation*}
\begin{split}
    &\sup_{0\leq t \leq T}\|u_\varepsilon(t,\cdot)\|^2_s+\omega(\eps)^{k} \int_0^T \|\langle x \rangle^{-{N/2}}\Lambda^{s+1/2}u_\eps(t,\cdot)\|_0^2\,dt\\
    & \hspace{4cm} \leq Ce^{C_1\omega(\eps)^{-k_1}T}\left(\|u_{0,\eps}\|_s^2+\int_0^T \| \langle x \rangle^{N/2}\Lambda^{s-1/2}f_\eps(t,\cdot)\|^2_0 \,dt\right),
    \end{split}
\end{equation*}
\end{itemize}  
where $C,C_1>0$ and $k$, $k_1 \in \mathbb{N}_0$, depend on the net of coefficients of $(P_\eps)_\eps$ and on $s$, but are independent of $\eps$.

\end{Rem}

\begin{Rem}\label{rmk.Hs}
    Let us emphasise that, in the case of $u_0 \in H^s(\R^n)$ and $f$ having suitable regularity (as required in $(i)$, $(ii)$ or $(iii)$ of the previous remark), we have that the net of solutions $(u_\eps)_\eps$ satisfies the estimates $(i)$, $(ii)$ and $(iii)$ with $u_0$ and $f$ appearing on the right-hand sides instead of their regularised versions $(f_\eps)_\eps$ and $(u_{0,\eps})_\eps$ (cf. Theorem \ref{theo.apriori}). It therefore follows that in this case, without regularising right-hand side and initial data, we can generate a very weak solution of $H^s$-type rather than $H^\infty$.
\end{Rem}

\subsection{Uniqueness}\label{sub.uniq}
We now prove a uniqueness result of a $H^\infty$-very weak solution in the following sense (cf. \cite{AACG}): if we perturb the coefficients of the regularised operator by suitable negligible nets, then the
net of solution of the Cauchy problem associated with the perturbed operator will differ from $(u_\eps)_\eps$ by a $H^\infty$-negligible net (recall Definition \ref{def.negHs}).

For this purpose, we perturb the coefficients and define the net of operators $(P'_\eps)_\eps$ by
\begin{equation}\label{eq.Pepsprimo}
\begin{split}
    P_\eps'&=D_t-A'_\eps-B'_\eps-V'_\eps\\
    &:=D_t-\sum_{i,j=1}^nD_{x_i}\bigl((a_{ij,\varepsilon}(x)+n_{ij,\eps}(x))D_{x_j}\bigr)  -\sum_{k=1}^n(b_{k,\varepsilon}(x)+n_{k,\eps}(x))D_{x_k}-(V_\varepsilon(x)+n_{V,\eps}(x)),
\end{split}
\end{equation}
where 
\begin{itemize}
    \item The matrix $\mathcal{N}_\eps(x)=(n_{ij,\eps}(x))_{i,j}$ is real and symmetric for all $x \in \R^n$. Furthermore, $n_{ij,\eps} \in C^\infty(\R^n)$ and for all $q \in \N_0$ and all $\beta \in \N_0^n$ there exists $C>0$ such that 
    \begin{equation*}
        \sup_{x \in \R^n}|\partial_x^{\beta}n_{ij,\eps}(x)|\leq C\eps^{q}, \quad \forall i,j=1,\dots,n,
    \end{equation*}
    for all $\eps \in (0,1]$. Finally, for each fixed $i,j,k=1,\dots,n$, we assume
\begin{equation*}
|\partial_{x_k} n_{ij,\varepsilon}(x)| \leq \nu \langle x \rangle^{-N}, \quad \forall x \in \R^n, \ \forall \eps \in (0,1],
\end{equation*}
where $\nu>0$ and $N>1$ are the constants appearing in \ref{hp3}.
\item For each $\beta \in \mathbb{N}_0^n$ and for each $q \in \N_0$ there exists a constant $C'>0$ such that for any $k=1,\dots,n$
\begin{equation*}
    \sup_{x \in \R^n}|\partial_x^\beta n_{k,\eps}(x)|\leq C'\eps^{q}, \quad  \forall \eps \in (0,1].
\end{equation*}
Moreover, there exists a universal constant $c_0'>0$ such that, for each $\varepsilon \in (0,1]$ (with $N>1$ as in \ref{hp3}),
\begin{equation*}
|\mathrm{Im}(n_{k,\varepsilon})(x)| \leq c_0' \langle x \rangle^{-N}, \quad \forall x \in \R^n, \ \forall k = 1, \dots, n.
\end{equation*}

\item 
For each $\beta \in \N_0^n$ and each $q \in \mathbb{N}_0$, there exists a constant $C''>0$ such that
\begin{equation*}
\sup_{x \in \R^n}|\partial_x^\beta n_{V,\varepsilon}(x)| \leq C\eps^q, \quad \forall \eps \in (0,1].
\end{equation*}
\end{itemize}
Furthermore, we consider a $H^\infty(\R^n)$-negligible perturbation of the Cauchy data in the sense of Definition \ref{def.negHs}. Specifically, we consider the Cauchy problem 
\begin{equation}\label{eq.cauchyPprimo}
\begin{cases}
P_\eps' u_\eps'=f_\eps+n_{f,\eps} \quad  \text{in} \ (0,T] \times \R^n, \\
u_\eps'(0,\cdot)=u_{0,\eps}+n_{u_0,\eps} \quad \text{in}\  \R^n,
\end{cases}
\end{equation}
where $P_\eps'$ is defined as in \eqref{eq.Pepsprimo}, $(f_\eps)$ and $(u_{0,\eps})_\eps $ are $H^\infty$-moderate nets and $(n_{f,\eps})_\eps$, $(n_{u_0,\eps})_\eps$ are $H^\infty$-negligible nets.

To ensure the existence of a solution of the problem \eqref{eq.cauchyPprimo}, we need the following proposition.
\begin{Prop}\label{prop.eps0}
    Let $(P_\eps')_\eps$ be the net of operators defined as in \eqref{eq.Pepsprimo}. Then, \ref{hp1}, \ref{hp2}, \ref{hp3}, \ref{hp4}, \ref{hp5} hold for $(P_\eps')_\eps$, provided $\eps \in (0,\eps_0]$, with $\eps_0>0$ sufficiently small.
\end{Prop}
\begin{proof}
    In this framework condition \ref{hp1}, \ref{hp3}, \ref{hp4} and \ref{hp5} are trivially satisfied. Therefore, we only need to check \ref{hp2}. 
    To do that we note that (recall $\mathcal{A}_\eps(x)=(a_{ij,\eps}(x))_{i,j}$ and $\mathcal{N}_\eps(x)=(n_{ij,\eps})_{i,j})$) 
    \begin{equation*}
        |(\mathcal{A}_\eps(x)+\mathcal{N}_\eps(x))\cdot \xi | \geq |\mathcal{A}_\eps(x) \cdot \xi|-|\mathcal{N}_\eps(x) \cdot \xi| \geq {\mu}^{-1} |\xi|-c'\eps^q|\xi| \geq c |\xi|, \quad \forall (x,\xi) \in \R^n \times \R^n,
    \end{equation*}
    for some $c',c>0$, independent of $\eps$, provided $\eps \in (0,\eps_0]$. 
\end{proof}
Thus, since $(P_\eps')_\eps$ satisfies the aforementioned conditions there exists $(u_\eps')_\eps$ $H^\infty$-very weak solution of \eqref{eq.cauchyPprimo}. Then, for each $\eps \in (0,\eps_0]$, with $\eps_0$ given as in Proposition \ref{prop.eps0}, we have
\begin{equation*}
\begin{cases}
    &P_\eps(u_\eps-u_\eps')=-n_{f,\eps}-(P_\eps-P_\eps')u_\eps'\\
    &u_\eps(0,\cdot)-u_\eps'(0,\cdot)=-n_{u_0,\eps}.
\end{cases}
\end{equation*}
Therefore, by \eqref{eq.i} we have that, for each $s \in \R$, there exist $C>0$ and $N \in \mathbb{N}$ so that
\begin{equation*}
     \sup_{0\leq t\leq T}\|u_\varepsilon(t,\cdot)-u_\varepsilon'(t,\cdot)\|_s\leq C\eps^{-N}\left(\|n_{u_0,\eps}\|_s+\int_0^T \|n_{f,\eps}(t,\cdot)+(P_\eps-P_\eps')u_\eps'(t,\cdot)\|_sdt\right),
\end{equation*}
for all $\eps \in (0,\eps_0]$. Hence, since $(n_{u_0,\eps})_\eps$ and $(n_{f,\eps})_\eps$ are $H^\infty$-negligible, and the net of coefficients of $(P_\eps-P_\eps')_\eps$ satisfies negligible estimates (cf. \cite{AACG}) we may conclude that for all $q>0$ there exists $c>0$ such that 
\begin{equation*}
    \sup_{0\leq t\leq T}\|u_\varepsilon(t,\cdot)-u_\varepsilon'(t,\cdot)\|_s\leq c\eps^{q}, \quad \forall \eps \in (0,\eps_0].
\end{equation*}
This proves the uniqueness result we wanted to obtain.

\section{Consistency with the classical theory}
In this section, we establish consistency with the classical theory. This means that when the coefficients are smooth, we recover the classical solution as $\eps \rightarrow 0$. We will consider first a problem with diagonal principal part and then we will deal with the general set-up.

We begin by recalling the classical result proven in \cite{FT} and \cite{KPRV} for evolution equation of the form 
\begin{equation}
\begin{cases}\label{mainprobclass}
    Pu=f \quad \text{in} \ ]0,T] \times \R^n, \\
    u(0,\cdot)=u_0 \quad \text{in} \ \R^n,
\end{cases}
\end{equation}
where $u_0 \in H^s(\R^n)$, for $s \in \R$, f has suitable regularity (see Theorem \ref{theo.class} below) and
\begin{equation}\label{eq.Pclass}
    P=D_t-\sum_{i,j=1}^nD_{x_i}(a_{ij}(x)D_{x_j})- \sum_{k=1}^nb_k(x)D_{x_k}-V(x),
\end{equation}
with $a_{ij},b_k,V \in C_b^\infty(\R^n)$ that satisfy the following hypotheses:
\begin{itemize}
    \item The matrix of the coefficients $\mathcal{A}(x)$ is real and symmetric for each $x \in \R^n$, and satisfies the non degeneracy condition: there exists $\mu>0$ such that 
    \begin{equation}\label{eq.condclass1}
        \mu^{-1}|\xi| \leq |\mathcal{A}(x) \cdot \xi| \leq \mu |\xi|, \quad \forall (x,\xi) \in \R^n \times \R^n.
\end{equation}
\item There exist $\nu>0$ sufficiently small and $N \in \mathbb{N}$, $N>1$, such that for all $\alpha \in \mathbb{N}_0^n$, with $|\alpha|\geq 1$, one has 
\begin{equation}\label{eq.condclass2}
|\partial_{x}^{\alpha} a_{ij}(x)| \leq \nu \langle x \rangle^{-N}, \quad \forall x \in \R^n, \ \ \forall i,j=1,\dots,n. 
\end{equation}
\item 
There exists a universal constant $c_0>0$ such that,
\begin{equation}\label{eq.condclass3}
|\mathrm{Im}(b_k)(x)| \leq c_0 \langle x \rangle^{-N}, \quad \forall x \in \R^n, \ \forall k = 1, \dots, n,
\end{equation}
with $N\in \mathbb{N}$ as above, and for each $\beta \in \mathbb{N}_0^n$ 
\begin{equation}\label{eq.condclass4}
    |\partial_x^\beta b_k(x)|\leq C_\beta, \quad \forall x \in \R^n, \ \forall k = 1, \dots, n,
\end{equation}
for some constant $C_\beta>0$.
\item For each $\gamma \in \mathbb{N}_0^n$ 
\begin{equation}\label{eq.condclass5}
|\partial_x^\gamma V(x)| \leq C_\gamma, \quad \forall x \in \R^n,
\end{equation}
for some constant $C_\gamma>0$
\end{itemize}

\begin{Th}\label{theo.class}
    Let $u_0 \in H^s(\R^n)$, and let $N>1$ as above.  Then, under the aforementioned hypotheses on $P$, we have the following results.   
    \begin{itemize}
    \item [(i)] If $f\in L^1([0,T];H^s(\R^n))$, the IVP \eqref{mainprobclass} has a unique solution $u \in C([0,T];H^s(\R^n))$ satisfying
    
    \begin{equation}\label{eq.iclass}
    \sup_{0\leq t\leq T}\|u(t,\cdot)\|_s\leq C_2e^{TC_1}\left(\|u_0\|_s+\int_0^T \|f(t,\cdot)\|_sdt\right).
\end{equation}
\item [(ii)]  If $f\in L^2([0,T];H^s(\R^n))$, the IVP \eqref{mainprobclass} has a unique solution $u\in C([0,T];H^{s}(\R^n))$ satisfying
\begin{equation}\label{eq.iiclass}
\begin{split}
    &\sup_{0\leq t \leq T}\|u(t,\cdot)\|^2_s+\int_0^T \|\langle x \rangle^{-{N/2}}\Lambda^{s+1/2}u(t,\cdot)\|_0^2\,dt\\
    &\hspace{4cm} \leq C_2e^{TC_1}\left(\|u_0\|_s^2+\int_0^T \|f(t,\cdot)\|_s^2 dt\right).
\end{split}
\end{equation}
\item [(iii)]  If $\Lambda^{s-\frac{1}{2}}f\in L^2([0,T]\times \R^n;\langle x \rangle^N dx dt)$, the IVP \eqref{mainprobclass} has a unique solution $u\in C([0,T];H^{s}(\R^n))$ satisfying
\begin{equation}\label{eq.iiiclass}
\begin{split}
    &\sup_{0\leq t \leq T}\|u(t,\cdot)\|^2_s+ \int_0^T \|\langle x \rangle^{-{N/2}}\Lambda^{s+1/2}u(t,\cdot)\|_0^2\,dt\\
    & \hspace{4cm} \leq C_2e^{TC_1}\left(\|u_0\|_s^2+\int_0^T \| \langle x \rangle^{N/2}\Lambda^{s-1/2}f(t,\cdot)\|^2_0 \,dt\right),
    \end{split}
\end{equation}
\end{itemize}  
for some constants $C_1,C_2>0$ depending on $s$ and on the coefficients of $P$. 
\end{Th}

\begin{Rem}\label{rmk.class}
The proof of this theorem is the same as in \cite{KPRV} and \cite{FT}. The result follows from a priori estimates, established, for instance, in Lemma 2.3.1 of \cite{KPRV}, together with a standard functional analysis argument. The key tool for obtaining such estimates is the so-called Doi's lemma (see, for instance, Lemma 2.2.2 in \cite{KPRV}), which in \cite{KPRV} is derived from a non-trapping condition on the bicharacteristic curves of the principal symbol of $A$. In the present setting, the non-trapping condition is replaced by the smallness condition \eqref{eq.condclass2}, under which Doi's lemma still holds (cf. \cite{FT}). Hence, the proof of the a priori estimates, which (once again together with a standard functional analysis argument) leads to the proof of Theorem \ref{theo.class}, is the same as that in Lemma 2.3.1 of \cite{KPRV}.
\end{Rem}

The next goal is to prove that whenever $P$ is an operator with smooth coefficients satisfying suitable conditions (which, in turn, imply conditions \eqref{eq.condclass1}--\eqref{eq.condclass5} and therefore guarantee well-posedness and the smoothing effect in the classical sense), the corresponding net of regularised operators$(P_\varepsilon)_\varepsilon$ satisfies \ref{hp1}--\ref{hp5}. By repeating the proof of Theorem \ref{theo_with_sc} we get the following Proposition.
\begin{Prop}\label{prop.Psmooth}
Let $P$ be an operator of the form 
\[
  P=D_t-\sum_{i,j=1}^nD_{x_i}(a_{ij}(x)D_{x_j})- \sum_{k=1}^nb_k(x)D_{x_k}-V(x),
\]
and assume that the coefficients fulfill the following hypotheses.
\begin{enumerate}[label=(\roman*)]
    \item The matrix $\mathcal{A}(x)=(a_{ij}(x))_{i,j=1}^n$ is real and symmetric, for all $ x \in \R^n$ and its entries can be decomposed as
    \[
    a_{ij}(x)=c_{ij}+\tilde{a}_{ij}(x), \quad  i,j=1,\dots,n,
    \]
    where $c_{ij} \in \R$ and $\tilde{a}_{ij} \in C_b^\infty(\R^n)$, for all $i,j=1,\dots,n$. Moreover, we assume that the matrix  $\mathcal{C}=(c_{ij})_{i,j=1}^n$ is non-degenerate, namely there exists $\mu'>0$ such that  
    \[
    {\mu'}^{-1}|\xi|\leq |\mathcal{C} \cdot \xi |\leq \mu' |\xi|, \quad \forall \xi \in \R^n,
    \]
    and for all $i,j=1,\dots,n$ the perturbations satisfy
    \begin{equation*}
    |\partial_x^{\alpha}\tilde{a}_{ij}(x)|\leq \nu' \langle x \rangle^{-N},\quad  \forall |\alpha|\leq 1, \ \forall x \in \R^n,
    \end{equation*}
    where $N \in \mathbb{N}$, $N>1$ and $\nu'>0$ is sufficiently small.
    \item For all $k=1,\dots,n$, we assume $b_k \in C^\infty(\R^n)$
    and (with $N>1$ as above)
    \[
    |\mathrm{Im}(b_k)(x)|\leq c_0'\langle x \rangle^{-N},
    \]
     for some universal constant $c_0'>0$.
    \item $V \in C^\infty_b(\R^n)$. 
\end{enumerate}
Then, the corresponding net of regularised operators $(P_\eps)_\eps$ defined as in Section \ref{sec.stateandreg}, fulfills \ref{hp1}-\ref{hp5}.
\end{Prop}
\begin{Rem}\label{rmk.cons}
    Note that if $P$ is an operator that satisfies the hypotheses of Proposition \ref{prop.Psmooth}, then it satisfies \eqref{eq.condclass1}-\eqref{eq.condclass5}.
\end{Rem}

\begin{Rem}\label{rmk.diag}
It is remarkable that the consistency result can also be applied to a simpler class of (regular) operators tailored to Example \ref{ex.ultradiag}.  
More precisely, let $P$ be the operator defined in \eqref{eq.Pclass} with diagonal matrix $\mathcal{A}$, i.e.
\[
  P=D_t-\sum_{i=1}^nD_{x_i}(a_{ii}(x)D_{x_i})- \sum_{k=1}^nb_k(x)D_{x_k}-V(x).
\]
We assume that  
\[
a_{ii}(x)=c_{i}+\tilde{a}_{ii}(x),
\]
where $c_{i}\in \mathbb{R} \setminus \lbrace 0\rbrace$ is a constant and $\tilde{a}_{ii}$ is smooth and bounded for $i=1,\cdots, n$ with 
\[
0\le \tilde{a}_{ii}(x)\le \frac{|c_i|}{2},
\]
for all $x$. Furthermore, we require that the lower order parts of $P$ satisfy the same hypotheses $(ii)$ and $(iii)$ of Proposition \ref{prop.Psmooth}.

Let $(P_\varepsilon)_\varepsilon$ be the corresponding regularised operators as in Section \ref{sec.stateandreg}. If $P$ fulfills the hypotheses \eqref{eq.condclass1}-\eqref{eq.condclass5} then $(P_\varepsilon)_\varepsilon$ fulfills the hypotheses $(H1)-(H5)$.

The proof of this fact follows by repeating the proof of Theorem \ref{theo_with_sc} in the case where $P$ is an operator with regular coefficients.
\end{Rem}

We are now ready to prove the following consistency result.

Let $P$ be an operator satisfying the hypotheses of Proposition \ref{prop.Psmooth} and let $f \in C([0,T];H^\infty(\R^n))$ and $u_0 \in H^\infty(\R^n)$. Let also $(f_\eps)_\eps$ and $(u_{0,\eps})_\eps$, with  $f_\eps(t,x)=(f(t,\cdot) \ast \varphi_\eps)(x)$ and $u_{0,\eps}(x)=(u_0 \ast\varphi_\eps)(x)$, for $\eps \in (0,1]$ and $\varphi \in \mathscr{S}(\R^n)$, $\int \varphi \, dx=1$, with all vanishing moments, i.e. $\int x^\alpha \varphi(x)dx=0$, $\forall \alpha \neq 0$.

We denote by $u$ the solution of 
\begin{equation}
    \begin{cases}
        P u=f \quad \text{in} \ \ (0,T] \times \R^n \\
        u(0,\cdot)=u_0\quad \text{in}\ \ \R^n,
    \end{cases}
\end{equation}
that we recall it exists and is unique due to Theorem \ref{theo.class} and Remark \ref{rmk.class}, and by $u_\eps$ the solution of 
\begin{equation}
    \begin{cases}
        P_\eps u_\eps=f_\eps \quad \text{in} \ \ (0,T] \times \R^n \\
        u_\eps(0,\cdot)=u_{0,\eps}\quad \text{in}\ \ \R^n.
    \end{cases}
\end{equation}

\begin{Prop}
\label{prop_consistency}
Under the hypotheses above any very weak solution $u_\eps$ converges to the classical solution $u$ in  $C([0,T];H^\infty(\R^n))$.
\end{Prop}
\begin{proof}
Note that $u-u_\eps$ solves, 
\begin{equation}
    \begin{cases}
        P(u-u_\eps)=(f-f_\eps)+R_\eps u_\eps \quad \text{in} \ \ (0,T] \times \R^n \\
        (u-u_\eps)(0,\cdot)=u_0-u_{0,\eps}\quad \text{in}\ \ \R^n,
    \end{cases}
\end{equation}
with 
\begin{equation}\label{eq.reps}
R_\eps=P_\eps-P=-\sum_{i,j=1}^nD_{x_i}\bigl((a_{ij,\eps}(x)-a_{ij}(x))D_{x_j}\bigr)- \sum_{k=1}^n(b_{k,\eps}(x)-b_k(x))D_{x_k}-(V_\eps(x)-V(x)).
\end{equation}
Therefore, denoting  $R_\eps=\mathrm{Op}(r_\eps)$, by Theorem \ref{theo.bound}, Theorem \ref{theo.class} and Remark \ref{rmk.cons} we have 
\begin{equation}
\begin{split}
    \sup_{0\leq t\leq T}\|u(t,\cdot)-u_\eps(t,\cdot)\|_s & \leq C\Bigl(\|u_0-u_{0,\eps}\|_s+\int_0^T \|(f-f_\eps)(t,\cdot)\|_sdt+ \int_0^T \|R_\eps  u_\eps(t,\cdot)\|_sdt \Bigr)\\
    & \leq C'\Bigl(\|u_0-u_{0,\eps}\|_s+\int_0^T \|(f-f_\eps)(t,\cdot)\|_sdt+ |r_\eps|^{(2)}_k\int_0^T \|  u_\eps(t,\cdot)\|_{s+2}dt \Bigr),
\end{split}
\end{equation}
for some $C,C',k$ independent of $\eps$.
We now note that, by reasoning as in \cite{AACG} we have that, since the coefficients of the operator $P$ are $C^\infty$, we obtain estimate \eqref{eq.i} uniformly with respect to $\eps$, that lead to  
\[
\|u_\eps(t,\cdot)\|_{s+2}\leq C_s,
\]
for some $C_s>0$, uniformly in $\eps$.
Moreover, since $R_\eps=\mathrm{Op}(r_\eps)$ is given by \eqref{eq.reps}, we easily obtain that 
\[
|r_\eps|_k^{(2)} \rightarrow 0, \quad \text{as} \ \eps \rightarrow 0^+.
\]
Finally, $(f-f_\varepsilon)\varepsilon$ and $(u_0-u_{0,\varepsilon})_\varepsilon$ are $H^\infty$-negligible nets, since the Cauchy data $u_0$ and $f$ have been regularised by means of a mollifier with all vanishing moments (cf. \cite{AACG}).

\noindent Therefore, in conclusion
\[
u_\eps \rightarrow u \quad \text{as} \ \eps \rightarrow 0, \quad \text{in} \ \ C([0,T];H^\infty(\R^n)),
\]
and this proves the consistency of our approach with the classical theory. 
\end{proof}
\begin{Rem}
 We conclude this work by emphasising that our approach allows us to obtain a convergence result as $\varepsilon \to 0$ even when the coefficients of the operator $P$, generating problem \eqref{mainprobIntro}, are merely $C^k$ in the spatial variables (rather than $C^\infty$, as required in the classical theory), for a sufficiently large integer $k$. The Cauchy data $f$ and $u_0$ are assumed to belong to the Sobolev space $H^s(\mathbb{R}^n)$ for some $s \in \mathbb{R}$.

More precisely, consider the Cauchy problem
\begin{equation*}
\begin{cases}
Pu = f & \text{in } (0,T] \times \mathbb{R}^n, \\
u(0,\cdot) = u_0 & \text{in } \mathbb{R}^n,
\end{cases}
\end{equation*}
where $u_0 \in H^s(\mathbb{R}^n)$, $f \in C([0,T]; H^s(\mathbb{R}^n))$, and $P$ is defined as in \eqref{def.P}, with coefficients in $C^k(\mathbb{R}^n)$ for some $k$, satisfying the structural conditions stated above.

By repeating the regularisation argument used in the proof of Theorem~\ref{theo.apriori}, we may construct a net of solutions $(u_\varepsilon)_\varepsilon$ solving the corresponding regularised problems induced by the net $(P_\varepsilon)_\varepsilon$, with the same initial data $u_0$ and right-hand side $f$. Moreover, if the coefficients of $P$ are sufficiently regular (i.e. if $k$ is sufficiently large) all the seminorms of symbols that appear in the proof of Theorem \ref{theo.apriori} actually do not depend on $\eps$. Therefore, by repeating the proof, we obtain that
\[
\sup_{0 \le t \le T} \|u_\varepsilon(t,\cdot)\|_{H^s}
\le
C_2 e^{C_1 T}
\left(
\|u_0\|_{H^s}
+
\int_0^T \|f(t,\cdot)\|_{H^s}\, dt
\right),
\]
uniformly in $\eps$. Hence, the net $(u_\varepsilon(t,\cdot))_\varepsilon$ is uniformly bounded in $\varepsilon$ in $H^s(\R^n)$, for each $t \in [0,T]$ and then it admits a subsequence converging weakly in this space.
\end{Rem}


\end{document}